\documentclass[a4paper,centertags,oneside,12pt]{amsart}
\usepackage{mathrsfs}
\usepackage{appendix}
\usepackage{amssymb}
\usepackage{fancyhdr}
\usepackage{charter}
\usepackage{typearea}
\usepackage{pdfsync}
\usepackage[utf8]{inputenc}
\usepackage[colorlinks,linkcolor=blue,anchorcolor=blue,citecolor=green]{hyperref}
\usepackage{mathrsfs}
\usepackage[a4paper,top=3cm,bottom=3cm,left=3cm,right=3cm]{geometry}

\usepackage{enumitem}
\usepackage{tikz}
\usepackage{enumitem}
\usepackage{color}
\setlist[1]{itemsep=5pt}
\newcommand{\comment}[1]{}

\makeatletter
      \def\@setcopyright{}
      \def\serieslogo@{}
      \makeatother

\newtheorem{theorem}{Theorem}[section]
\newtheorem{lemma}[theorem]{Lemma}
\newtheorem{corollary}[theorem]{Corollary}

\newtheorem{proposition}[theorem]{Proposition}

\newtheorem{definition}[theorem]{Definition}

\numberwithin{equation}{section}

\begin{document}
\title{Hardy Spaces and Toeplitz operators on the Pentablock}

\author{Guicong Su}
\address{School of Mathematics and Statistics, Wuhan University of Technology, Wuhan, Hubei 430070, P.
R. China}
\thanks{* Corresponding author.}
\email{suguicong@whut.edu.cn}

\author{Nitin Tomar*}
\address{Statistics and Mathematics Unit, Indian Statistical Institute, Bangalore,
Karnataka – 560059, India.}
\email{tomarnitin414@gmail.com, nitin\_pd@isibang.ac.in}

\thanks{ 
\textperiodcentered \; Brown-Halmos characterization \textperiodcentered \; The commutant lifting theorem \textperiodcentered \; Uniqueness set \textperiodcentered \; Pentablock contractions}

\begin{abstract} The pentablock is a domain in $\mathbb C^3$ arising in the study of the
$\mu$-synthesis problem and is defined by
\[
\mathcal P
=
\left\{
(a_{21},\operatorname{tr}(A),\det(A)):
A=[a_{ij}]_{i,j=1}^2\in M_2(\mathbb C),\ \|A\|<1
\right\}.
\]
We introduce the Hardy space $H^2(\mathcal P)$ associated with the
canonical measure $\mu$ on the distinguished boundary $b\mathcal P$.
We identify $H^2(\mathcal P)$ with a closed subspace of
$L^2(b\mathcal P,\mu)$ and show that it is a reproducing kernel Hilbert
space. We establish a new description of $\mathcal P$ and its closure
and obtain a uniqueness set. In particular,  we prove that $\mathcal P$ is the interior of the polynomial convex hull of $\Psi(\overline{\mathbb D}^3)$, where
\[
\Psi(\xi,\lambda_1,\lambda_2)
=
\left(
\frac{\xi(\lambda_1-\lambda_2)}{2},
\lambda_1+\lambda_2,
\lambda_1\lambda_2
\right)
\]
is the parametrization map for $\mathcal P$. We also show that
$H^2(\mathcal P)$ embeds isometrically as a closed subspace of
$H^2(\mathbb D^3)$. We then obtain a Brown-Halmos type
characterization of Toeplitz operators on $H^2(\mathcal P)$ and a
commutant lifting theorem for $\mathcal P$-isometries and their minimal
$\mathcal P$-unitary extensions. As consequences, we determine the
commutant of the coordinate multipliers, characterize the compact
Toeplitz operators and show that the coordinate multipliers have no
nontrivial joint reducing subspaces. We further introduce two weighted Hardy spaces on $\mathcal P$ associated with the complex Jacobian of $\Psi$. One is closely related to the Hardy space of the symmetrized bidisc. The other is associated with the Jacobian-squared density from the change-of-variables formula. Finally, we characterize $\mathcal P$-unitaries as images of commuting triples of unitaries under $\Psi$ and present a class of $\mathcal P$-contractions from triples having
$\overline{\mathbb D}^3$ as a spectral set.
\end{abstract}

\maketitle

\tableofcontents

\section{Introduction}

\noindent The class of Toeplitz operators on function spaces constitute one of the most foundational operator classes in modern function and operator theory. On the classical Hardy space $H^2(\mathbb{D})$ on the unit disc $\mathbb D=\{z \in \mathbb C: |z|<1\}$, a bounded operator $T$ is Toeplitz if and only if it satisfies the Brown-Halmos relation
\[
    M_z^* T M_z = T,
\]
where $M_z$ denotes multiplication by the coordinate function. This algebraic characterization, first established by Brown and Halmos \cite{BrownHalmos1964}, has lead to further study of analogous criteria for the Toeplitz operators on Hardy spaces over domains in several complex variables. Such multivariable generalizations have since been obtained for the polydisc \cite{Ding2003, DingSunZheng2012, MajiSarkarSarkar2018}, the Euclidean unit ball \cite{DavieJewell1977}, the Hartogs triangle \cite{JainPramanick2025}, the symmetrized polydisc \cite{BhattacharyyaDasSau2021, DasSau2021} and the tetrablock \cite{JainKumarMalPramanick2026}. The Toeplitz operators on bounded symmetric domains, Cartan domains and their proper images have been studied in \cite{Upmeier1983, KumarMalPramanick2025, GhoshShyamRoy2024}. These seminal works suggest that the geometry of the underlying domain, its distinguished boundary and the canonical coordinate multipliers jointly determine the appropriate notion of a Toeplitz operator and the corresponding Brown-Halmos type relations. The primary aim of this paper is to develop the corresponding theory for the pentablock $\mathcal P$, with the construction of its Hardy space as the starting point. Recall that 
\[
 \mathcal{P}
 =
 \left\{
   \bigl(a_{21},\operatorname{tr} A,\det A\bigr):
   A=[a_{ij}]_{i,j=1}^2,\ \|A\|<1
 \right\}
 \subseteq \mathbb{C}^3 .
\]
Introduced by Agler, Lykova and Young \cite{AglerLykovaYoung2015} in the context of $\mu$-synthesis problem, the pentablock $\mathcal{P}$ is a non-convex, polynomially convex and $(1, 1, 2)$-quasi-balanced domain. The domain $\mathcal P$ has attracted considerable attention in recent years, owing to its rich geometric and operator-theoretic structure, e.g., see \cite{Als, Suguicong, SuTuWang2015, Kosinski2015, Jin, JindalKumar2024, PalTomar2026}. It is closely related to two earlier domains arising in the same context, namely, the symmetrized bidisc \cite{AglerYoung1999} and the tetrablock \cite{AbouhajarWhiteYoung2007} given by
\begin{align*}
\mathbb G_2&=\{(tr(A), \det(A)) : A=[a_{ij}]_{i, j=1}^2 \in M_2(\mathbb C), \|A\|<1\} \quad \text{and}\\
\mathbb E&=\{(a_{11}, a_{22}, \det(A)) : A=[a_{ij}]_{i, j=1}^2 \in M_2(\mathbb C), A=A^t, \|A\|<1\},
\end{align*}
respectively. 
The construction of Hardy spaces on domains arising as proper holomorphic
images has been investigated previously for the symmetrized bidisc \cite{MisraShyamRoyZhang2013} and the
tetrablock \cite{BeraChavanJain2025}. The construction of a Hardy space on $\mathcal P$ requires a different approach from those for $\mathbb G_2$ and $\mathbb E$, which we describe in Section \ref{sec_02}. Also, the Hardy space construction for bounded symmetric domains is not applicable to $\mathcal P$, as the pentablock is not a bounded symmetric domain. A key feature of $\mathcal{P}$ is that its distinguished boundary $b \mathcal P$ (see \cite{AglerLykovaYoung2015, Jin}) can be described as the image of $3$-torus $\mathbb T^3$ under the map $\Psi$ given by
\[
 \Psi(\xi,\lambda_1,\lambda_2)
 =
 \left(
   \frac{\xi(\lambda_1-\lambda_2)}{2},
   \lambda_1+\lambda_2,
   \lambda_1\lambda_2
 \right).
 \]
The paramaterization map $\Psi$ is central to our construction of the Hardy space $H^2(\mathcal P)$ on the pentablock. In fact, we prove in Section \ref{sec_03} that 
\[
\widehat{\Psi(\overline{\mathbb{D}}^3)}=\overline{\mathcal{P}} \quad \text{and} \quad \text{int}\left(\widehat{\Psi(\overline{\mathbb{D}}^3)}\right)=\mathcal{P},
\]
where $\widehat{K}$ and $\text{int}(K)$ denote the polynomial convex hull and interior of a set $K$, respectively. Furthermore, it is proved in Section \ref{sec_03} that $\Psi(\mathbb D^3)$ is a uniqueness set for $\mathcal P$ in the sense that if $f, g \in \text{Hol}(\mathcal P)$ with $f|_{\Psi(\mathbb D^3)}=g|_{\Psi(\mathbb D^3)}$, then $f=g$ on $\mathcal{P}$. These properties make $\Psi$ a natural parametrization for transferring the boundary measure from $\mathbb T^3$ to $b\mathcal P$. We define in Section \ref{sec_02} the canonical probability measure $\mu$ on $b\mathcal P$ as the pushforward of normalized Haar measure on $\mathbb T^3$ under $\Psi$ and set
\[
H^2(\mathcal P)
=
\overline{\mathbb C[a,s,p]}^{\,L^2(b\mathcal P,\mu)}.
\]
The resulting space admits two natural orthogonal decompositions, one induced by $(1, 1, 2)$-quasi-circular action $(a,s,p)\mapsto(\lambda a,\lambda s,\lambda^2p)$ for $\lambda\in \overline{\mathbb D}$ and the other arising from the $a$-fibres. Capitalizing on these orthogonal decompositions, we first establish that the boundary construction gives a Hardy space of holomorphic functions on $\mathcal{P}$. We prove that every $f\in H^2(\mathcal{P})$ admits a holomorphic realization that converges locally uniformly on $\mathcal{P}$, and that the point evaluations are bounded. Consequently, $H^2(\mathcal{P})$ is a reproducing kernel Hilbert space. We describe the $H^2(\mathcal{P})$-norm using the $(1,1,2)$-dilation maps
$\delta_r(a,s,p) = (ra, rs, r^2p)$ for $0<r<1$ and prove the isometric isomorphism
\[
H^2(\mathcal{P}) = \left\{ f\in\operatorname{Hol}(\mathcal{P}) : \sup_{0<r<1} \int_{b\mathcal{P}} |f\circ\delta_r|^2\,\mathrm{d}\mu < \infty \right\}.
\]
It follows that every $f\in H^2(\mathcal{P})$ has a unique radial boundary value in $L^2(b\mathcal{P},\mu)$. In Section \ref{sec_03}, we identify $H^2(\mathcal{P})$ with a closed subspace of the classical Hardy space $H^2(\mathbb{D}^3)$ via the map $f \mapsto f \circ \Psi$. We further determine the multiplier algebra of $H^2(\mathcal{P})$, proving the isometric equality
$
\operatorname{Mult}\bigl(H^2(\mathcal{P})\bigr) = H^\infty(\mathcal{P})$ with  $\|M_\varphi\| = \|\varphi\|_{\infty,\mathcal{P}}$ for all $\varphi\in H^\infty(\mathcal{P})$. Finally, we show that every bounded holomorphic function on $\mathcal{P}$ has an essentially bounded radial boundary value, providing an isometric algebra embedding of $H^\infty(\mathcal{P})$ into $L^\infty(b\mathcal{P},\mu)$.

\smallskip 

In Section \ref{sec_04}, we develop Toeplitz operator theory on $H^2(\mathcal{P})$. For $\varphi\in L^\infty(b\mathcal{P},\mu)$, we define the Toeplitz operator
\[
    T_\varphi = P_{H^2(\mathcal{P})} M_\varphi\big|_{H^2(\mathcal{P})}.
\]
The coordinate multipliers $(T_a,T_s,T_p)$ form a canonical $\mathcal{P}$-isometry on $H^2(\mathcal{P})$ and the corresponding multiplication triple $(M_a,M_s,M_p)$ on $L^2(b\mathcal{P},\mu)$ is its minimal normal $\mathcal{P}$-unitary extension. The operator identities satisfied by $\mathcal P$-isometries (see \cite{JindalKumar2024, PalTomar2026}) suggest considering the identities for $(T_a,T_s,T_p)$ given by
\[
T_p^*TT_p=T,\qquad
T_s^*TT_p=TT_s,\qquad
T_a^*TT_a+\frac14T_s^*TT_s=T,
\]
where $T\in\mathcal B(H^2(\mathcal P))$. We prove that these relations fully characterize Toeplitz operators on $H^2(\mathcal{P})$, so they serve as the Brown-Halmos relations for $\mathcal P$. A key step in the proof is a commutant-lifting argument based on the spherical normalization of the minimal $\mathcal{P}$-unitary extension and the fixed-point principle of Davie and Jewell \cite{DavieJewell1977}. More precisely, a bounded operator $T$ on $H^2(\mathcal P)$ satisfying the Brown-Halmos relations has a norm-preserving lift to an operator on $L^2(b \mathcal P, \mu)$ that commutes with the coordinate multiplication operators. This lift  must be a multiplication operator, which gives the required Toeplitz representation of $T$. As corollaries, it follows that every bounded operator commuting with $T_a,T_s,T_p$ is Toeplitz, and that the zero operator is the only compact Toeplitz operator on $H^2(\mathcal{P})$. Moreover, we prove that the coordinate multipliers $T_a, T_s, T_p$ admit no nontrivial joint reducing subspaces. Finally, we formulate this lifting argument for an arbitrary $\mathcal{P}$-isometry and its minimal $\mathcal{P}$-unitary extension, giving a commutant lifting theorem. In Section \ref{sec_05}, we consider two weighted Hardy spaces on
$\mathcal P$ arising from the complex Jacobian of the parametrization
$\Psi$. The weights $|J_\Psi|$ and $|J_\Psi|^2$ lead to the strict chain
of continuous inclusions
\[
H^2(\mathcal P)\subsetneq H_J^2(\mathcal P)
\subsetneq H_{J^2}^2(\mathcal P).
\]
The first weighted space $H_J^2(\mathcal P)$ is naturally related to the
Hardy space $H^2(\mathbb G_2)$ of the symmetrized bidisc through the
embedding $g\mapsto g\circ\pi_{\mathrm{symm}}$, where
$\pi_{\mathrm{symm}}(z_1,z_2)=(z_1+z_2,z_1z_2)$ is the symmetrization map.
In particular, $H^2(\mathbb G_2)$ embeds isometrically into
$H_J^2(\mathcal P)$. On the other hand, for a holomorphic map $F$ between domains of the same complex dimension, the real Jacobian is the squared modulus of the complex Jacobian, i.e.,
$\det D_{\mathbb R}F=|\det F'|^2=|J_F|^2$ so that $|J_F|^2$ is the volume factor in the change-of-variables formula. This provides the geometric motivation for the second weighted
space $H_{J^2}^2(\mathcal P)$ associated with $|J_\Psi|^2$. 

\smallskip 

The parametrization of the pentablock by the map $\Psi$ also provides a natural framework for studying operator tuples associated with $\mathcal P$. In Section \ref{sec_06}, we develop this connection at the operator level. We first show that the boundary representation $b\mathcal P=\Psi(\mathbb T^3)$ admits a operator-theoretic analogue, i.e., a commuting triple is a $\mathcal P$-unitary if and only if it can be represented as $\Psi(U_1,U_2,U_3)$ for a commuting triple of unitaries $U_1,U_2,U_3$. This gives a natural way of constructing $\mathcal P$-isometries from commuting triples of isometries. We then show that the same construction applied to a commuting triple having $\overline{\mathbb D}^3$ as a spectral set produces a $\mathcal P$-contraction, and that a unitary dilation of the original triple induces a $\mathcal P$-unitary dilation of the resulting triple.


\section{The Hardy space on the pentablock}\label{sec_02}

\noindent In this section, we introduce the Hardy space associated with the pentablock. In principle, there are several possible ways of associating a Hardy-type space with the pentablock. We begin by recalling the constructions of Hardy spaces on the symmetrized
bidisc and the tetrablock.  A common feature of these constructions is the
presence of a proper holomorphic map from a classical domain of the same
complex dimension onto the domain under consideration. 

\smallskip 

Recall that the symmetrized bidisc is given by
$\mathbb G_2=\pi_{\mathrm{symm}}(\mathbb D^2)$, where
$\pi_{\mathrm{symm}}:\mathbb C^2\to\mathbb C^2$ denotes the symmetrization map
defined by $\pi_{\mathrm{symm}}(z_1,z_2)=(z_1+z_2,z_1z_2)$. Its closure and
distinguished boundary are given by
$\Gamma_2=\pi_{\mathrm{symm}}(\overline{\mathbb D}^2)$ and
$b\Gamma_2=\pi_{\mathrm{symm}}(\mathbb T^2)$, respectively. The map $\pi_{\mathrm{symm}}$ is a proper holomorphic map of multiplicity two with $J_{\pi_{\mathrm{symm}}}(z_1,z_2)=z_1-z_2$. The authors of \cite{MisraShyamRoyZhang2013} introduced the Hardy space
$H^2(\mathbb G_2)$ as the space consisting of all holomorphic maps $g$ on $\mathbb G_2$ such that
\[
\|g\|_{H^2(\mathbb G_2)}^2
=
\frac{1}{2}
\sup_{0<r<1}
\int_{\mathbb T^2}
\left|g\circ\pi_{\mathrm{symm},r}(\lambda_1, \lambda_2)
\right|^2
\left|J_{\pi_{\mathrm{symm}}}(r\lambda_1,r\lambda_2)\right|^2
\,dm_2 < \infty,
\]
where $\pi_{\mathrm{symm},r}(\lambda_1,\lambda_2)
=
\left(r(\lambda_1+\lambda_2),r^2\lambda_1\lambda_2\right)$ and $m_2$ is the normalized Haar measure on $\mathbb T^2$. A similar construction from \cite{BeraChavanJain2025} applies to the tetrablock $\mathbb E$.  Let
\[
\mathcal R_{II}
=
\left\{
(z_1,z_2,z_3)\in\mathbb C^3:
\left\|
\begin{pmatrix}
z_1&z_3\\
z_3&z_2
\end{pmatrix}
\right\|<1
\right\}
\]
and consider $\varphi:\mathcal R_{II}\to \mathbb E$ defined as $
\varphi(z_1,z_2,z_3)
=
(z_1,z_2,z_1z_2-z_3^2)$. The map $\varphi$ is proper and its source and target have the same complex
dimension. Moreover, $J_\varphi(z_1,z_2,z_3) =-2z_3$. The corresponding Hardy space is given by
\[
H^2(\mathbb E)
=
\left\{
f\in\operatorname{Hol}(\mathbb E):
\|J_\varphi f\circ\varphi\|_{H^2(\mathcal R_{II})} <\infty 
\right\}.
\]
Hence, in both cases, the Hardy space on the target domain is obtained by pulling back to a domain of the same complex dimension and incorporating the
complex Jacobian of the proper holomorphic map. However, $\mathcal P$ does not admit a direct construction of this form from its natural matrix realization given by
\[
    \pi: \left\{A=[a_{ij}] \in M_2(\mathbb C): \|A\|<1  \right\} \to \mathcal P,
    \qquad
    \pi(A)=\bigl(a_{21},\operatorname{tr}(A),\det(A)\bigr)
\]
We recall from \cite{AglerLykovaYoung2015} that the extension of $\pi$ to the distinguished boundary $U(2)=\{A \in M_2(\mathbb C): A^*A=I\}$ of the normed unit ball in $M_2(\mathbb C)$ has image
\[
    K_1
    =
    \left\{
        (a,s,p)\in\mathcal P:
        |a|^2+\frac{|s|^2}{4} \leq 1, \ 
        (s,p)\in b\Gamma_2
    \right\},
\]
where $b\Gamma_2=\pi_{\mathrm{symm}}(\mathbb T^2)$ is the distinguished boundary of $\mathbb G_2$. One could therefore use the push-forward of the normalized Haar measure on $U(2)$ to define an $L^2$-space on $K_1$.  However, we shall not adopt this approach since $b\mathcal P\subsetneq K_1$.  Consequently, a boundary $L^2$-space obtained from the push-forward of Haar measure on $U(2)$ is, in general, a different boundary construction from one supported on $b\mathcal P$.  Since the Hardy space is intended here to reflect the distinguished boundary structure of the domain, we choose the latter.  This preserves the analogy with the Hardy space constructions for $\mathbb G_2$ and $\mathbb E$, in which the underlying boundary measure is supported on their distinguished boundaries. Also, note that the map
\[
    \pi:M_2(\mathbb C)\simeq\mathbb C^4\longrightarrow\mathbb C^3
\]
has source and target of different complex dimensions and hence there is no holomorphic Jacobian determinant.  Thus the matrix realization of $\mathcal P$ does not provide a direct Jacobian-weighted analogue of the corresponding constructions for $\mathbb E$. To overcome this, we consider a particularly useful parametrization of the distinguished boundary $b \mathcal P$ of the pentablock.  Recall from \cite{AglerLykovaYoung2015,Jin} that
\[
b\mathcal{P}
=
\left\{
(a,s,p)\in\mathbb{C}^3:
(s,p)\in b\Gamma_2,\;
|a|^2=1-\frac{|s|^2}{4}
\right\}.
\]
It is not difficult to see that $b\mathcal P$ admits the torus parameterization given by
\[
b \mathcal P=
\left\{\left(
\frac{\xi(\lambda_1-\lambda_2)}{2},
\lambda_1+\lambda_2,
\lambda_1\lambda_2
\right):
(\xi,\lambda_1,\lambda_2)\in\mathbb{T}^3\right\},
\]
This parameterization is generically $2$-to-$1$, reflecting the symmetry \((\xi,\lambda_1,\lambda_2)\mapsto(-\xi,\lambda_2,\lambda_1)\). Furthermore, let
$
\Psi:\mathbb D^3\to \mathbb C
$
be the holomorphic map given by
\[
\Psi(\xi,\lambda_1,\lambda_2)
=
\left(
\frac{\xi(\lambda_1-\lambda_2)}{2},
\lambda_1+\lambda_2,
\lambda_1\lambda_2
\right).
\]
The map $\Psi$ extends continuously to $\overline{\mathbb D}^3$ and $\Psi(\mathbb{T}^3)=b\mathcal{P}$. Let $m_k$ denote normalized Haar measure on the $k$-torus $\mathbb{T}^k$ and define the pushforward measure $\mu=\Psi_*m_3$.
Equivalently, for every \(F\in\mathcal{C}(b\mathcal{P})\),
\[
\int_{b\mathcal{P}}F(a,s,p)\,\mathrm{d}\mu
=
\int_{\mathbb{T}^3}
F\!\left(
\frac{\xi(\lambda_1-\lambda_2)}{2},
\lambda_1+\lambda_2,
\lambda_1\lambda_2
\right)
\,\mathrm{d}m_3(\xi,\lambda_1,\lambda_2).
\]
Thus, the map $\Psi$ gives a natural parametrization of the distinguished boundary itself, providing the boundary measure used in the definition above. To make the structure of the measure $\mu$ explicit, we first describe the
fibres of the natural projection from $b\mathcal{P}$ onto $b\Gamma_2$. Define
$ \operatorname{pr}:b\mathcal{P}\to b\Gamma_2$ as
    $\operatorname{pr}(a,s,p)=(s,p)$. For $(s,p)\in b\Gamma_2$, put
\[
    \rho(s,p)
    =
    \sqrt{1-\frac{|s|^2}{4}}.
\]
Since $|a|^2=1-|s|^2\slash 4$
on $b\mathcal{P}$, the fibre of $\operatorname{pr}$ over $(s,p)$ is
\[
    \operatorname{pr}^{-1}(s,p)
    =
    \left\{
        (a,s,p)\in b\mathcal{P}: |a|=\rho(s,p)
    \right\} =
    \left\{
        \bigl(\xi\rho(s,p),s,p\bigr):
        \xi\in\mathbb{T}
    \right\}.
\]
Thus, for each fixed $(s,p)\in b\Gamma_2$, the $a$-coordinate ranges over
a circle of radius $\rho(s,p)$. In particular, $\rho(s,p)$ is the radius
of the fibre over $(s,p)$. If $|s|=2$, then $\rho(s,p)=0$ and the
corresponding fibre degenerates to the single point $(0,s,p)$. This description gives a fibre decomposition of the measure $\mu$. More
precisely, if $\nu$ denotes the canonical probability measure on $b\Gamma_2$,
then
\[
    \int_{b\mathcal{P}} F(a,s,p)\,d\mu
    =
    \int_{b\Gamma_2}
    \int_{\mathbb{T}}
    F\bigl(\xi\rho(s,p),s,p\bigr)
    \,dm_1(\xi)\,d\nu(s,p),
    \qquad
    F\in C(b\mathcal{P}).
\]
Thus, on each fibre over $(s,p)$, the measure induced by $\mu$ is the
normalized rotationally invariant measure obtained from Haar measure on
$\mathbb{T}$ under the parametrization $\xi\longmapsto \xi\rho(s,p)$. This description immediately gives the invariance of $\mu$ under rotations
of the $a$-fibres. For $\eta\in\mathbb{T}$, define
$R_\eta(a,s,p)=(\eta a,s,p)$. For a point on the fibre over $(s,p)$, we have that
$
    R_\eta\bigl(\xi\rho(s,p),s,p\bigr)
    =
    \bigl((\eta\xi)\rho(s,p),s,p\bigr)$. Thus, $R_\eta$ fixes the base point $(s,p)$ and rotates the angular
parameter $\xi$ on the corresponding fibre. Since $m_1$ is normalized Haar
measure on $\mathbb{T}$,
\[
    \int_{\mathbb{T}} G(\eta\xi)\,dm_1(\xi)
    =
    \int_{\mathbb{T}} G(\xi)\,dm_1(\xi)
\]
for every $G\in L^1(\mathbb{T})$. Consequently,
\[
\begin{aligned}
    \int_{b\mathcal{P}} F\circ R_\eta\,d\mu
    =
    \int_{b\Gamma_2}
    \int_{\mathbb{T}}
    F\bigl(\eta\xi\rho(s,p),s,p\bigr)
    \,dm_1(\xi)\,d\nu(s,p)
    =
    \int_{b\mathcal{P}} F\,d\mu.
\end{aligned}
\]
In terms of pushforward measures, we have that $
    (R_\eta)_*\mu=\mu$ for every $\eta\in\mathbb{T}$. We next consider the $(1,1,2)$-quasi-circular action
$q_\omega(a,s,p)=(\omega a,\omega s,\omega^2p)$, where $\omega\in\mathbb{T}$. Since $\rho(\omega s,\omega^2p)
    =\rho(s, p)$, it follows that 
    \[
    q_\omega\bigl(\xi\rho(s,p),s,p\bigr)
    =
    \bigl(\omega\xi\rho(s,p),\omega s,\omega^2p\bigr)
    =
    \bigl(
        (\omega\xi)\rho(\omega s,\omega^2p),
        \omega s,\omega^2p
    \bigr).
\]    
Therefore, $q_\omega$ maps the fibre over $(s,p)$ onto the fibre over
$(\omega s,\omega^2p)$ and, on the angular parameter of the fibre, acts
by $\xi\mapsto\omega\xi$. The invariance of $\mu$ under $q_\omega$ is now an immediate consequence of the original torus parametrization, because
\[
\begin{aligned}
    \Psi(\xi,\omega\lambda_1,\omega\lambda_2)
    &=
    \left(
        \frac{\xi\omega(\lambda_1-\lambda_2)}{2},
        \omega(\lambda_1+\lambda_2),
        \omega^2\lambda_1\lambda_2
    \right)
   =
    q_\omega\bigl(\Psi(\xi,\lambda_1,\lambda_2)\bigr).
\end{aligned}
\]
The map
$
    (\xi,\lambda_1,\lambda_2)
    \mapsto
    (\xi,\omega\lambda_1,\omega\lambda_2)
$
preserves normalized Haar measure on $\mathbb{T}^3$. Since
$\mu=\Psi_*m_3$, it follows that $(q_\omega)_*\mu=\mu$. Equivalently,
\[
    \int_{b\mathcal{P}}
    F(\omega a,\omega s,\omega^2p)\,d\mu
    =
    \int_{b\mathcal{P}}
    F(a,s,p)\,d\mu,
    \qquad
    F\in C(b\mathcal{P}).
\]
Thus the measure $\mu$ is invariant under two natural symmetries of
$b\mathcal{P}$. The first symmetry rotates each circular $a$-fibre while
leaving its base point $(s,p)$ fixed. The second carries the fibre over
$(s,p)$ onto the fibre over $(\omega s,\omega^2p)$, preserving its radius
and rotating its angular parameter. These invariance properties follow,
respectively, from the rotational invariance of Haar measure on
$\mathbb{T}$ and from the invariance of Haar measure on $\mathbb{T}^3$
under the corresponding torus rotation. We now define the boundary Hardy space associated with this measure.

\begin{definition}
The boundary Hardy space of $\mathcal P$ is
$
H^2(\mathcal{P})
=
\overline{\mathbb{C}[a,s,p]}^{\,L^2(b\mathcal{P},\mu)}
$
equipped with the inner product
\[
\langle f,g\rangle_{H^2(\mathcal{P})}
=
\int_{b\mathcal{P}}f(a,s,p)\,\overline{g(a,s,p)}\,\mathrm{d}\mu.
\]
\end{definition}

The two invariance properties of $\mu$ established above provide two natural orthogonal decompositions of $H^2(\mathcal{P})$. The invariance
under the $(1,1,2)$-quasi-circular action leads to a decomposition
according to weighted homogeneous degree, while the rotational invariance
of the $a$-fibres yields a decomposition according to the powers of the first variable $a$. We describe these two decompositions in turn.

\smallskip

For $n\in\mathbb{N}_0$, let $\mathcal{Q}_n$ denote the space of weighted
homogeneous polynomials of degree $n$, i.e., $q(\omega a,\omega s,\omega^2p)
    =
    \omega^n q(a,s,p)$ for every $\omega\in\mathbb{T}$. More precisely,
\[
    \mathcal{Q}_n
    =
    \operatorname{span}
    \left\{
        a^j s^k p^\ell:
        j,k,\ell\geq0,\;
        j+k+2\ell=n
    \right\}.
\]
The invariance of \(\mu\) under the quasi-circular action implies \(\mathcal{Q}_n\perp\mathcal{Q}_m\) for \(n\neq m\). Since polynomials are dense by definition, we have the orthogonal decomposition
\begin{equation}\label{orth}
H^2(\mathcal{P})
=
\bigoplus_{n=0}^{\infty}\mathcal{Q}_n.
\end{equation}
Moreover, the dimension formula reads
\[
\dim\mathcal{Q}_n
=
\sum_{\ell=0}^{\lfloor n/2\rfloor}(n-2\ell+1)
=
\left\lfloor\frac{(n+2)^2}{4}\right\rfloor.
\]
The second decomposition originates from
the Hartogs fibre structure in
the $a$-variable. For fixed $(s,p)\in b\Gamma_2$, the corresponding fibre is
$
    \left\{
        \bigl(\xi\rho(s,p),s,p\bigr):
        \xi\in\mathbb{T}
    \right\}$. Thus, integration over the fibre is integration with respect to Haar
measure in the angular variable $\xi$. In particular, for
$j,k\in\mathbb{Z}_+$,
\[
\begin{aligned}
    \int_{\mathbb{T}}
    \bigl(\xi\rho(s,p)\bigr)^k
    \overline{\bigl(\xi\rho(s,p)\bigr)^j}
    \,dm_1(\xi)  =
    \rho(s,p)^{k+j}
    \int_{\mathbb{T}}\xi^{k-j}\,dm_1(\xi)
    =
    \begin{cases}
        \rho(s,p)^{2k}, & j=k,\\[4pt]
        0, & j\neq k.
    \end{cases}
\end{aligned}
\]
Consequently, distinct powers of the fibre variable $a$ are orthogonal
in $L^2(b\mathcal{P},\mu)$. This also explains the weights that occur on the symmetrized bidisc. Indeed, if $f\in\mathbb{C}[s,p]$, then
\[
\begin{aligned}
    \|a^k f\|_{H^2(\mathcal{P})}^2
   =
    \int_{b\Gamma_2}\int_{\mathbb{T}}
    \rho(s,p)^{2k}|f(s,p)|^2
    \,dm_1(\xi)\,d\nu(s,p)    =
    \int_{b\Gamma_2}
    |f(s,p)|^2\rho(s,p)^{2k}\,d\nu(s,p).
\end{aligned}
\]
Thus the natural coefficient space corresponding to the $k$-th power of
the fibre variable is the weighted Hardy space
\[
    H_k^2(\mathbb{G}_2)
    =
    \overline{\mathbb{C}[s,p]}^{
        \,L^2(b\Gamma_2,\rho^{2k}\,d\nu)
    } \quad \text{with norm}
\quad
    \|f\|_{H_k^2(\mathbb{G}_2)}^2
    =
    \int_{b\Gamma_2}
    |f(s,p)|^2\rho(s,p)^{2k}\,d\nu(s,p).
\]
Since $\rho(s,p)>0$ for $\nu$-almost every $(s,p)\in b\Gamma_2$, this
defines a norm on $\mathbb{C}[s,p]$. Every polynomial $q\in\mathbb{C}[a,s,p]$ has a unique  decomposition
\[
    q(a,s,p)
    =
    \sum_{k=0}^N a^k q_k(s,p),
    \qquad \text{and so,} \qquad    \|q\|_{H^2(\mathcal{P})}^2
    =
    \sum_{k=0}^N
    \|q_k\|_{H_k^2(\mathbb{G}_2)}^2, 
\]
where  $q_k\in\mathbb{C}[s,p]$. Motivated by this, we present the following result. 

\begin{proposition}\label{prop:decompositions}
The Hardy space $H^2(\mathcal{P})$ admits the two orthogonal
decompositions
\begin{equation}\label{eq:quasi-decomposition}
    H^2(\mathcal{P})
    =
    \bigoplus_{n=0}^{\infty}\mathcal{Q}_n=
    \bigoplus_{k=0}^{\infty}
    a^k H_k^2(\mathbb{G}_2),
\end{equation}
where $\mathcal{Q}_n$ is the space of weighted homogeneous polynomials
of degree $n$. Moreover, every $f\in H^2(\mathcal{P})$ has a unique expansion
$f(a,s,p) =
    \overset{\infty}{\underset{k=0}{\sum}}a^k f_k(s,p)$ for  $f_k\in H_k^2(\mathbb{G}_2)$ with convergence in $H^2(\mathcal{P})$ and
\[
    \|f\|_{H^2(\mathcal{P})}^2
    =
    \sum_{k=0}^{\infty}
    \|f_k\|_{H_k^2(\mathbb{G}_2)}^2.
\]
\end{proposition}

\begin{proof}
By above discussion, it suffices to prove the expansion of $f$ as in the statement. Let $f\in H^2(\mathcal P)$. Since the polynomials are dense in $H^2(\mathcal P)$, there exists a sequence of polynomials $\{q^{(n)}\}_{n\geq1}$ such that $q^{(n)}\to f$ in $H^2(\mathcal P)$. For each $n$, write
\[
q^{(n)}(a,s,p)
=
\sum_{k=0}^{N_n}a^kq_k^{(n)}(s,p),
\qquad
q_k^{(n)}\in\mathbb C[s,p],
\]
and set $q_k^{(n)}=0$ for $k>N_n$. By the orthogonality of distinct powers of $a$ and the definition of
$H_k^2(\mathbb{G}_2)$, we have that
\[
\left\|
\sum_{k=0}^{\infty}a^kq_k^{(n)}
\right\|_{H^2(\mathcal P)}^2
=
\sum_{k=0}^{\infty}
\|q_k^{(n)}\|_{H_k^2(\mathbb{G}_2)}^2
\ \ \text{and } \ 
\|q^{(n)}-q^{(m)}\|_{H^2(\mathcal P)}^2
=
\sum_{k=0}^{\infty}
\|q_k^{(n)}-q_k^{(m)}\|_{H_k^2(\mathbb{G}_2)}^2.
\]
Since $\{q^{(n)}\}$ converges in $H^2(\mathcal P)$, it is a Cauchy
sequence. Thus $\{q_k^{(n)}\}_{n\geq1}$ is Cauchy in $H_k^2(\mathbb{G}_2)$. Since
$H_k^2(\mathbb{G}_2)$ is complete, there exists $f_k\in H_k^2(\mathbb{G}_2)$
such that $q_k^{(n)}
\to f_k$ in $H_k^2(\mathbb{G}_2)$. Next, we show that
$
\sum_{k=0}^{\infty}a^kf_k
$
converges in $H^2(\mathcal P)$. Let $C>0$ be such that $\|q^{(n)}\|_{H^2(\mathcal P)}^2\leq C$ for $n \geq 1$. Fix $N \in \mathbb N$. Then
\[
\sum_{k=0}^{N}
\|q_k^{(n)}\|_{H_k^2(\mathbb{G}_2)}^2
\leq C \quad \text{and so,} \quad \sum_{k=0}^{N}
\|f_k\|_{H_k^2(\mathbb{G}_2)}^2
\leq C.
\]
Since $N$ is arbitrary, it follows that 
\[
\sum_{k=0}^{\infty}
\|f_k\|_{H_k^2(\mathbb{G}_2)}^2<\infty
\quad \text{and thus} \quad 
\left\|
\sum_{k=N+1}^{M}a^kf_k
\right\|_{H^2(\mathcal P)}^2
=
\sum_{k=N+1}^{M}
\|f_k\|_{H_k^2(\mathbb{G}_2)}^2
\to 0
\]
as $M,N\to\infty$. Since $H^2(\mathcal P)$ is complete, there exists $g \in H^2(\mathcal P)$ such that $g=
\sum_{k=0}^{\infty}a^kf_k$. It remains to identify $g$ with $f$. Fix $n\geq1$. Since
$q_k^{(m)}\to f_k$ in $H_k^2(\mathbb G_2)$ for every $k$, we have by Fatou's lemma that
\[
\begin{aligned}
\sum_{k=0}^{\infty}
\|f_k-q_k^{(n)}\|_{H_k^2(\mathbb G_2)}^2
\leq
\liminf_{m}
\sum_{k=0}^{\infty}
\|q_k^{(m)}-q_k^{(n)}\|_{H_k^2(\mathbb G_2)}^2
&=
\liminf_{m}
\|q^{(m)}-q^{(n)}\|_{H^2(\mathcal P)}^2\\
&=
\|f-q^{(n)}\|_{H^2(\mathcal P)}^2
\end{aligned}
\]
for every $n \geq 1$. Also, the orthogonality of the summands in the series representation of $g$ and $q^{(n)}$ gives
\[
\|g-q^{(n)}\|_{H^2(\mathcal P)}^2
=
\sum_{k=0}^{\infty}
\|f_k-q_k^{(n)}\|_{H_k^2(\mathbb G_2)}^2 \quad \text{and so,} \quad 
\|g-q^{(n)}\|_{H^2(\mathcal P)}
\leq
\|f-q^{(n)}\|_{H^2(\mathcal P)}
\to 0
\]
as $n \to \infty$. Then  $\|f-g\|_{H^2(\mathcal P)}
\leq \|f-q^{(n)}\|_{H^2(\mathcal P)}
+
\|q^{(n)}-g\|_{H^2(\mathcal P)}
\to 0$. Hence $f=g$, and therefore
\[
f(a,s,p)
=
\sum_{k=0}^{\infty}a^kf_k(s,p)
\]
with convergence in $H^2(\mathcal P)$. The uniqueness of the expansion follows from the orthogonality of the
subspaces $a^kH_k^2(\mathbb G_2)$. Indeed, if
$
\sum_{k=0}^{\infty}a^kf_k=0,
$
then, for every $j\geq0$,
\[
0
=
\left\langle
\sum_{k=0}^{\infty}a^kf_k,a^jf_j
\right\rangle_{H^2(\mathcal P)}
=
\|f_j\|_{H_j^2(\mathbb G_2)}^2,
\]
and hence $f_j=0$. Thus the expansion is unique. Since the partial sums converge to $f$ and the summands are
orthogonal, it follows that
\[
\left\|
\sum_{k=0}^{N}a^kf_k
\right\|_{H^2(\mathcal P)}^2
=
\sum_{k=0}^{N}
\|f_k\|_{H_k^2(\mathbb G_2)}^2 \quad \text{and so,} \quad 
\|f\|_{H^2(\mathcal P)}^2
=
\sum_{k=0}^{\infty}
\|f_k\|_{H_k^2(\mathbb G_2)}^2.
\]
This completes the proof.
\end{proof}

Next, we show that each function in $H^2(\mathcal P)$ is holomorphic on $\mathcal P$ and the space $H^2(\mathcal P)$ is a reproducing kernel Hilbert space. To do so, we need the following lemma establishing an estimate that
will allow us to pass from the orthogonal decomposition in
$H^2(\mathcal{P})$ to a locally uniformly convergent holomorphic series.

\begin{lemma}\label{lemma1}
Let \(K\Subset\mathcal{P}\). There exist constants \(0<r_K<1\) and \(C_K>0\) such that
\[
\sup_{z\in K}|q(z)|
\leq
C_K(n+1)^{3/2}r_K^n\|q\|_{L^2(b\mathcal{P})}
\]
for every \(q\in\mathcal{Q}_n\).
\end{lemma}

\begin{proof}
Since $K\Subset\mathcal{P}$ and $\mathcal{P}$ is open, there exists $R>1$, sufficiently close to $1$, such that $\delta_R(K)\Subset\mathcal{P}$. Here, $\delta_r(a, s, p)=(ra, rs, r^2p)$ for $0<r<1$. Put $r_K=R^{-1}$. Since $q\in\mathcal{Q}_n$, its weighted homogeneity
gives that $
    q(\lambda a,\lambda s,\lambda^2p)
    =
    \lambda^n q(a,s,p)$ for every $\lambda\in\mathbb{C}$. Hence, for $z\in K$ and $w=\delta_R(z)$,
\[
    |q(z)|
    =
    R^{-n}|q(w)|
    =
    r_K^n|q(w)| 
    \leq
    r_K^n\sup_{\mathcal{P}}|q|
    \leq
    r_K^n\|q\|_{L^\infty(b\mathcal{P})}.
\]
It now remains to estimate the boundary supremum in terms of the boundary
$L^2$ norm. Let
$Q=q\circ\Psi$. Since $q\in\mathcal{Q}_n$, the function $Q$ is a trigonometric
polynomial on $\mathbb{T}^3$ involving at most $(n+1)^3$ Fourier
monomials. Writing
\[
    Q=\sum_{\alpha\in A}c_\alpha\chi_\alpha \quad \text{with} \quad |A|\leq(n+1)^3.
\]
Hence, the Cauchy-Schwarz inequality and orthogonality of the characters give
\[
\begin{aligned}
 \|q\|_{L^\infty(b \mathcal P)} =  \|Q\|_{L^\infty(\mathbb{T}^3)}
    \leq
    \sum_{\alpha\in A}|c_\alpha|
    \leq
    |A|^{1/2}
    \left(\sum_{\alpha\in A}|c_\alpha|^2\right)^{1/2}
    \leq
    (n+1)^{3/2}\|Q\|_{L^2(\mathbb{T}^3)}.
\end{aligned}
\]
Since $\mu=\Psi_*m_3$, we have that $\|Q\|_{L^2(\mathbb{T}^3)}=
    \|q\|_{L^2(b\mathcal{P},\mu)}$
and therefore
\[
    \sup_{z\in K}|q(z)|
    \leq
    (n+1)^{3/2}r_K^n
    \|q\|_{L^2(b\mathcal{P},\mu)}.
\]
This proves the lemma.
\end{proof}

The estimate in Lemma \ref{lemma1} gives sufficient decay in the weighted degree to sum the homogeneous components uniformly on compact subsets of $\mathcal{P}$. We therefore obtain the following holomorphic realization theorem of the Hardy space $H^2(\mathcal P)$.

\begin{theorem}[Holomorphic realization]\label{thm:hol-realization}
Let $f\in H^2(\mathcal P)$ and let
$f=\sum_{n=0}^{\infty}q_n$ be its weighted homogeneous decomposition with each $q_n \in \mathcal Q_n$. Then, for
every $z\in\mathcal P$, the series
$\sum_{n=0}^{\infty}q_n(z)$
converges absolutely and the convergence is uniform on compact subsets of $\mathcal P$. The resulting map
\[
    T:H^2(\mathcal P)\to \operatorname{Hol}(\mathcal P), \quad     (Tf)(z)=\sum_{n=0}^{\infty}q_n(z)
\]
is injective and continuous when $\operatorname{Hol}(\mathcal P)$ is equipped with the compact-open topology. Also, $H^2(\mathcal P)$ is a reproducing kernel Hilbert space.
\end{theorem}

\begin{proof}
Let $K$ be a compact subset of $\mathcal P$ and let $n \geq 0$. By Lemma \ref{lemma1}, there
exist constants $C_K>0$ and $0<r_K<1$ such that
$|q_n(z)|
    \leq
    C_K(n+1)^{3/2}r_K^n
    \|q_n\|_{L^2(b\mathcal P,\mu)}$ for every $z \in K$. Since the decomposition of $f$ is orthogonal in $H^2(\mathcal P)$, we have
\[
    \sum_{n=0}^{\infty}
    \|q_n\|_{L^2(b\mathcal P,\mu)}^2
    =
    \|f\|_{H^2(\mathcal P)}^2
    <\infty.
\]
By the Cauchy-Schwarz inequality,
\[
\begin{aligned}
    \sum_{n=0}^{\infty}C_K(n+1)^{3/2}r_K^n
    \|q_n\|_{L^2(b\mathcal P,\mu)}
    \leq C_K
    \left(
        \sum_{n=0}^{\infty}
        (n+1)^3r_K^{2n}
    \right)^{1/2}
    \left(
        \sum_{n=0}^{\infty}
        \|q_n\|_{L^2(b\mathcal P,\mu)}^2
    \right)^{1/2}.
\end{aligned}
\]
The first series on the right is finite since $0<r_K<1$, while the
second one equals $\|f\|_{H^2(\mathcal P)}^2$. By the Weierstrass $M$-test, the series
\[
    \sum_{n=0}^{\infty}q_n(z)
\]
converges absolutely and uniformly for $z\in K$. Since $K\Subset\mathcal
P$ was arbitrary, the convergence is locally uniform on $\mathcal P$. Each $q_n$ is a polynomial and hence holomorphic on $\mathcal P$.
Therefore, by the Weierstrass theorem on locally uniform limits of
holomorphic functions, the function
\[
    Tf(z)=\sum_{n=0}^{\infty}q_n(z)
\]
is holomorphic on $\mathcal P$. We next prove that $T$ is injective. Suppose that $Tf\equiv0$ on
$\mathcal P$. Fix $z=(a,s,p)\in\mathcal P$. Then, for every $\lambda\in\mathbb D$, we have
\[
\begin{aligned}
    0
    =Tf(\lambda a,\lambda s,\lambda^2p)
    =\sum_{n=0}^{\infty}
      q_n(\lambda a,\lambda s,\lambda^2p)
    =\sum_{n=0}^{\infty}
      \lambda^n q_n(a,s,p).
\end{aligned}
\]
The second equality is justified by the locally uniform convergence
already proved, while the third follows from the weighted homogeneity
of $q_n$. Consequently, $q_n\equiv 0$ on $\mathcal P$ for every $n$ and thus
$f=0$ in $H^2(\mathcal P)$ showing that $T$ is injective. For
$K\Subset\mathcal P$, there exists a constant $A_K>0$, depending only on $K$, such that
\[
    \sup_{z\in K}|Tf(z)|
    \leq
    A_K\|f\|_{H^2(\mathcal P)}.
\]
This proves that $T$ is continuous from $H^2(\mathcal P)$ into
$\operatorname{Hol}(\mathcal P)$ endowed with the compact-open topology, and that the point evaluations are bounded linear functionals on $H^2(\mathcal P)$. Thus $H^2(\mathcal P)$ is a reproducing kernel Hilbert space.
\end{proof}

The above theorem identifies each element of the boundary Hardy
space with a holomorphic function on $\mathcal{P}$. It remains to
characterize this holomorphic realization intrinsically in terms of radial boundary norms.  The $(1,1,2)$-quasi-balanced property of $\mathcal P$ provides the radial maps
\[
    \delta_r(a,s,p)=(ra,rs,r^2p), \qquad 0\leq r<1.
\]
It is not difficult to see that $\delta_r(\overline{\mathcal P}) \subseteq \mathcal P$ for every $r \in (0, 1)$. This leads naturally to the following radial
Hardy space:
\[
H^2_{\mathrm{rad}}(\mathcal{P})
=
\left\{
f\in\operatorname{Hol}(\mathcal{P}):
\|f\|_{H^2_{\mathrm{rad}}(\mathcal{P})}
<\infty
\right\},
\]
where
\[
\|f\|_{H^2_{\mathrm{rad}}(\mathcal{P})}=
\sup_{0<r<1}
\left(\int_{b\mathcal{P}}
|f\circ\delta_r|^2\,d\mu\right)^{1\slash 2}.
\]
We first verify that the above quantity defines a norm. Assume that $
    \|f\|_{H^2_{\mathrm{rad}}(\mathcal P)}=0$. Then, for every $r\in(0,1)$, we have that $f\circ\delta_r=0$ $\mu$-almost everywhere on $b\mathcal P$. Since $\mu=\Psi_*m_3$ and
$\Psi(\mathbb T^3)=b\mathcal P$, while normalized Haar measure $m_3$ has
full support on $\mathbb T^3$, the measure $\mu$ has full support on
$b\mathcal P$. Moreover, $f\circ\delta_r$ is continuous on
$\overline{\mathcal P}$. It follows that $f\circ\delta_r=0$ on $b \mathcal P$. Let us define 
\[
\mathcal{P}_r=\delta_r(\overline {\mathcal P})=\{(a, s, p): (r^{-1}a, r^{-1}s, r^{-2}p) \in \overline{\mathcal P}\}.
\]
The authors of \cite{AglerLykovaYoung2015} proved that $\mathcal P=\cup_{0<r<1}\mathcal P_r$. Fix $r\in(0,1)$ and set
$f_r =f\circ\delta_r$. Since $\delta_r(\overline{\mathcal P})\subset\mathcal P$, the function
$f_r$ is holomorphic on $\mathcal P$ and continuous on
$\overline{\mathcal P}$. We have already shown that
$f_r=0$ on $b\mathcal P$. Since $b\mathcal P$ is the Shilov boundary
of $\mathcal P$, it follows that  $f_r=0$ on $\overline{\mathcal P}$. Consequently, $f=0$ on $\mathcal P_r$ for every $r \in (0, 1)$ and thus, $f=0$ on $\mathcal P$. Hence, $\|f\|_{H^2_{\mathrm{rad}}(\mathcal{P})}=0$ if and only if $f=0$ for every $f \in \text{Hol}(\mathcal P)$. The absolute homogeneity and the triangle inequality follow from the
corresponding properties of the $L^2(b\mathcal P,\mu)$ norm. Thus $H^2_{\mathrm{rad}}(\mathcal P)$ is a normed space. We shall establish completeness after identifying this space with the boundary Hardy space
$H^2(\mathcal P)$. The next result provides this identification and at
the same time, it shows that the radial norm is precisely the boundary $H^2$-norm. 

\begin{theorem}[Radial realization]\label{thm:radial-realization}
The holomorphic realization identifies $H^2(\mathcal P)$ isometrically
with $H^2_{\mathrm{rad}}(\mathcal P)$. More precisely, if
$f=\sum_{n=0}^{\infty}q_n$ is the weighted homogeneous decomposition of $f$ in $H^2(\mathcal P)$ with each $q_n\in\mathcal Q_n$,
then
\[
    \int_{b\mathcal P}|f\circ\delta_r|^2\,d\mu
    =
    \sum_{n=0}^{\infty}
    r^{2n}\|q_n\|_{L^2(b\mathcal P,\mu)}^2 \ \ (0<r<1)
\ \ \text{and} \ \
    \|f\|_{H^2(\mathcal P)}^2
    =
    \sup_{0<r<1}
    \int_{b\mathcal P}|f\circ\delta_r|^2\,d\mu.
\]
Conversely, every $f\in H^2_{\mathrm{rad}}(\mathcal P)$ admits a unique expansion $f=\sum_{n=0}^{\infty}q_n$ with 
    $q_n\in\mathcal Q_n$ and
 the series converges locally uniformly on $\mathcal P$ and
\[
    \sum_{n=0}^{\infty}
    \|q_n\|_{L^2(b\mathcal P,\mu)}^2<\infty.
\]
\end{theorem}

\begin{proof}
We first consider $f\in H^2(\mathcal P)$ and write $f=\sum_{n=0}^{\infty}q_n$
as the weighted homogeneous decomposition with each $q_n\in\mathcal Q_n$. By Theorem
\ref{thm:hol-realization}, the series converges locally uniformly on
$\mathcal P$ to the holomorphic realization of $f$. Since
$\delta_r(\overline{\mathcal P})\Subset\mathcal P$ for $0<r<1$, it follows that 
$
    f\circ\delta_r
    =
    \sum_{n=0}^{\infty}r^nq_n
$
pointwise on $b\mathcal P$. Moreover, the series on the right converges in
$L^2(b\mathcal P,\mu)$. Indeed,
\[
    \sum_{n=0}^{\infty}
    \|r^nq_n\|_{L^2(b\mathcal P,\mu)}^2
    =
    \sum_{n=0}^{\infty}
    r^{2n}\|q_n\|_{L^2(b\mathcal P,\mu)}^2
    \leq
    \sum_{n=0}^{\infty}
    \|q_n\|_{L^2(b\mathcal P,\mu)}^2
    =
    \|f\|_{H^2(\mathcal P)}^2.
\]
Since the spaces $\mathcal Q_n$ are mutually orthogonal in
$L^2(b\mathcal P,\mu)$, it follows that
\[
\begin{aligned}
    \int_{b\mathcal P}|f\circ\delta_r|^2\,d\mu
    =
    \left\|
        \sum_{n=0}^{\infty}r^nq_n
    \right\|_{L^2(b\mathcal P,\mu)}^2
    =
    \sum_{n=0}^{\infty}
    r^{2n}\|q_n\|_{L^2(b\mathcal P,\mu)}^2.
\end{aligned}
\]
Since the terms on the right are nonnegative and increase with $r$, we have by monotone convergence theorem that
\[
\begin{aligned}
    \sup_{0<r<1}
    \int_{b\mathcal P}|f\circ\delta_r|^2\,d\mu
    =
    \lim_{r\uparrow1}
    \sum_{n=0}^{\infty}
    r^{2n}\|q_n\|_{L^2(b\mathcal P,\mu)}^2
    =
    \sum_{n=0}^{\infty}
    \|q_n\|_{L^2(b\mathcal P,\mu)}^2
    =
    \|f\|_{H^2(\mathcal P)}^2.
\end{aligned}
\]
Thus $f\in H^2_{\mathrm{rad}}(\mathcal P)$ and the holomorphic
realization preserves the norm. Conversely, let
$f\in H^2_{\mathrm{rad}}(\mathcal P)$. For $z=(a,s,p)\in\mathcal P$,
consider the holomorphic function $\lambda\mapsto
    f(\lambda a,\lambda s,\lambda^2p)$ for  $\lambda\in\mathbb D$. The $(1,1,2)$-quasi-balanced property of $\mathcal P$ guarantees that the above function is well-defined. Hence this function admits a Taylor expansion
\[
    f(\lambda a,\lambda s,\lambda^2p)
    =
    \sum_{n=0}^{\infty}\lambda^nq_n(a,s,p),
    \qquad \lambda\in\mathbb D,
\]
and by Cauchy's intergal formula,
\[
    q_n(a,s,p)
    =
    \int_{\mathbb T}
    f(\omega a,\omega s,\omega^2p)
    \overline{\omega}^{\,n}\,dm_1(\omega),
\]
where the integral may be taken over any circle
$|\lambda|=\rho<1$ after the corresponding rescaling. The Taylor
expansion of $f$ at the origin shows that $q_n$ is a polynomial
consisting precisely of those monomials $a^js^kp^\ell$ satisfying
$j+k+2\ell=n$. Consequently,
$q_n\in\mathcal Q_n$. We next establish the required square-summability of these
coefficients. Fix $0<r<1$ and define $F_r(a,s,p)=
    f(ra,rs,r^2p)$ for $(a,s,p)\in b\mathcal P$. By the preceding expansion,
\[
    F_r(a,s,p)
    =
    \sum_{n=0}^{\infty}
    r^nq_n(a,s,p).
\]
We claim that the functions
$
    \omega\mapsto
    F_r(\omega a,\omega s,\omega^2p)$ have Fourier coefficients $r^nq_n(a,s,p)$. Indeed, using the weighted
homogeneity of $q_n$,
\[
\begin{aligned}
    F_r(\omega a,\omega s,\omega^2p)
    =
    f(r\omega a,r\omega s,r^2\omega^2p)
    =
    \sum_{n=0}^{\infty}
    r^n\omega^nq_n(a,s,p).
\end{aligned}
\]
Since $\mu$ is invariant under the $(1,1,2)$-quasi-circular action, we have by the
Parseval's identity that
\[
 \sup_{0<r<1}
    \sum_{n=0}^{\infty}
    r^{2n}\|q_n\|_{L^2(b\mathcal P,\mu)}^2
    =\sup_{0<r<1}\int_{b\mathcal P}
    |f\circ\delta_r|^2\,d\mu
   =    \|f\|_{H^2_{\mathrm{rad}}(\mathcal{P})}
    <\infty.
\]
Letting $r\uparrow1$ and applying monotone convergence yields
\[
    \sum_{n=0}^{\infty}
    \|q_n\|_{L^2(b\mathcal P,\mu)}^2
    <\infty.
\]
Therefore
$
  z \mapsto  \sum_{n=0}^{\infty}q_n(z)
$ defines a function in $H^2(\mathcal P)$. By Theorem
\ref{thm:hol-realization}, its holomorphic realization is given by a locally uniformly convergent series on $\mathcal P$. It remains to identify this holomorphic function with the original
$f$. For every $z=(a, s, p) \in\mathcal P$, the one-variable Taylor expansion
above gives
\[
    f(\lambda a,\lambda s,\lambda^2p)
    =
    \sum_{n=0}^{\infty}\lambda^nq_n(a,s,p),
    \qquad |\lambda|<1.
\]
Since $z=(a, s, p) \in \mathcal P$, there exists $\rho>1$ sufficiently close to $1$ such that
$(\lambda a, \lambda s, \lambda^2 p) \in\mathcal P$ for $|\lambda|<\rho$. Thus
\[
    f(a,s,p)
    =
    \sum_{n=0}^{\infty}q_n(a,s,p).
\]
Hence the holomorphic realization of the element
$\sum_{n\geq0}q_n\in H^2(\mathcal P)$ coincides with the original
function $f$ on $\mathcal P$. The proof is now complete.
\end{proof}

As immediate consequences of Theorem \ref{thm:radial-realization}, we have the following corollaries.

\begin{corollary}[Radial approximation]\label{lem_rad}
For \(f\in H^2(\mathcal{P})\), let \(f_r=f\circ\delta_r\). Then
\[
\lim_{r \to 1}\|f_r-f\|_{H^2(\mathcal{P})}=0.
\]
\end{corollary}

\begin{proof}
Let \(f=\sum_{n\geq0}q_n\) be the weighted homogeneous decomposition of $f$ in $H^2(\mathcal P)$. Then
$
\|f_r-f\|_{H^2(\mathcal{P})}^2
=
\sum_{n=0}^{\infty}|r^n-1|^2\|q_n\|_{L^2(b\mathcal{P})}^2$,
which converges to $0$ by the  dominated convergence theorem.
\end{proof}

We next establish an identification of $H^2(\mathcal P)$ with a closed subspace of $L^2(b\mathcal P,\mu)$ via radial boundary values. In the present setting, this is an analogue of the classical realization of $H^2(\mathbb D)$ as a subspace of $L^2(\mathbb T)$. 

\begin{corollary}[Radial boundary values]\label{radial}
For every \(f\in H^2(\mathcal{P})\), there exists a unique \(f^*\in L^2(b\mathcal{P},\mu)\) such that
\[
\lim_{r\uparrow1}
\|f\circ\delta_r-f^*\|_{L^2(b\mathcal{P}, \mu)}=0.
\]
{In particular, $\|f\|_{H^2(\mathcal P)}=\|f^*\|_{L^2(b \mathcal P, \mu)}$.} 
\end{corollary}

\begin{proof}
Let $f=\sum_{n=0}^{\infty}q_n$ be the weighted homogeneous decomposition of $f$ with each $q_n \in \mathcal{Q}_n$. Each $q_n \in L^2(b \mathcal P, \mu)$ and we have by
Theorem \ref{thm:radial-realization} that
\[
    \sum_{n=0}^{\infty}
    \|q_n\|_{L^2(b\mathcal P,\mu)}^2
    =
    \|f\|_{H^2(\mathcal P)}^2<\infty.
\]
Since the spaces \(\mathcal Q_n\) are mutually orthogonal, the partial
sums
$
 s_N=   \sum_{n=0}^N q_n
$
form a Cauchy sequence in \(L^2(b\mathcal P,\mu)\). Hence there exists a
unique \(f^*\in L^2(b\mathcal P,\mu)\) such that $\{s_N\}$ converges to $f^*$ in $L^2(b\mathcal P,\mu)$. Moreover, for \(0<r<1\),
\[
    f\circ\delta_r
    =
    \sum_{n=0}^{\infty}r^nq_n
\]
in \(L^2(b\mathcal P,\mu)\). Therefore,
\[
\begin{aligned}
    \|f\circ\delta_r-f^*\|_{L^2(b\mathcal P,\mu)}^2
    &=
    \sum_{n=0}^{\infty}
    |r^n-1|^2
    \|q_n\|_{L^2(b\mathcal P,\mu)}^2,
\end{aligned}
\]
which converges to $0$ by the dominated convergence theorem. 
Note that 
\begin{align*}
    \|f\|^2_{H^2(\mathcal P)}=\underset{0<r<1}{\sup}\int_{b\mathcal P}|f \circ \delta_r|^2 d\mu=\underset{0<r<1}{\sup}\left[\sum_{n\geq0}r^{2n}\|q_n\|_{L^2(b\mathcal{P})}^2\right]
=\sum_{n\geq0}\|q_n\|_{L^2(b\mathcal{P})}^2
\end{align*}
and 
\begin{align*}
    \|f^*\|^2_{L^2(b \mathcal P, \mu)}=\underset{r \to 1}{\lim}\|f\circ \delta_r\|^2_{L^2(b \mathcal P, \mu)}=\underset{r \to 1}{\lim}\int_{b \mathcal P}|f \circ \delta_r|^2 d\mu
    &=\underset{r \to 1}{\lim}\left[\sum_{n\geq0}r^{2n}\|q_n\|_{L^2(b\mathcal{P})}^2\right]\\
&=\sum_{n\geq0}\|q_n\|_{L^2(b\mathcal{P})}^2.
\end{align*}
Hence, $\|f\|_{H^2(\mathcal P)}=\|f^*\|_{L^2(b \mathcal P, \mu)}$, which completes the proof.
\end{proof}

\section{A uniqueness set of $\mathcal P$ and the multiplier algebra of $H^2(\mathcal P)$}\label{sec_03}

\noindent
The holomorphic map
\[
\Psi(\xi,\lambda_1,\lambda_2)
=
\left(
\frac{\xi(\lambda_1-\lambda_2)}{2},
\lambda_1+\lambda_2,
\lambda_1\lambda_2
\right)
\]
has played a central role in our construction of the Hardy space
$H^2(\mathcal P)$.  In particular, the identity
\[
\Psi(\mathbb T^3)=b\mathcal P
\]
provides the distinguished boundary parametrization through which the
boundary measure on $\mathcal P$ and, consequently, the Hardy space
$H^2(\mathcal P)$ are defined.  The same parametrization also encodes
geometric information about the pentablock. In this section, we prove that $\overline{\mathcal P}$ is the polynomial convex hull of  $\Psi(\overline{\mathbb D}^{\,3})$ and thus $\mathcal P$ can be described as the interior of the polynomially convex
hull of $\Psi(\overline{\mathbb D}^{\,3})$. Also, we capitalize on these geometric and boundary realizations of
$\mathcal P$ to obtain a corresponding realization of $H^2(\mathcal P)$
inside the classical Hardy space $H^2(\mathbb D^3)$.  Composition with
$\Psi$ yields an isometric embedding, which allows us to study the
reproducing kernel and the evaluation functionals of $H^2(\mathcal P)$
within a familiar Hilbert-space setting.  This, in turn, leads naturally
to the operator-theoretic description of the bounded holomorphic
functions on $\mathcal P$. Our main result identifies the multiplier algebra of $H^2(\mathcal P)$
with $H^\infty(\mathcal P)$ isometrically, that is,
\[
\operatorname{Mult}(H^2(\mathcal P))
=
H^\infty(\mathcal P).
\]
Thus, boundedness as a holomorphic function on $\mathcal P$ is precisely
equivalent to boundedness of the corresponding multiplication operator
on $H^2(\mathcal P)$.  We finally relate this identification to the
boundary realization, showing that the boundary-value map embeds
$H^\infty(\mathcal P)$ isometrically into
$L^\infty(b\mathcal P,\mu)$. We begin with the following geometric properties of the map $\Psi$. Prior to this, we recall that a compact set $K \subseteq \mathbb C^n$ is said to be polynomially convex if it coincides with its polynomial convex hull $\widehat{K}$, where
\[
\widehat{K}=\left\{z \in \mathbb C^n: |p(z)| \leq \|p\|_{\infty, K} \ \text{for every polynomial $p$ in $n$-variables}\right\}.
\]
The authors of \cite{AglerLykovaYoung2015} proved that $\overline{\mathcal P}$ is polynomially convex. Our next result describes the domain $\mathcal P$ in terms of the polynomial convex hull of $\Psi(\overline{\mathbb D}^3)$.

\begin{proposition}\label{prop:Psi-interior}
The map \(\Psi\) satisfies
\[
\Psi(\mathbb{T}^3)=b\mathcal{P},
\qquad
\Psi(\overline{\mathbb{D}}\times\mathbb{D}^2)\subsetneq\mathcal{P},
\qquad
\Psi(\overline{\mathbb{D}}^3)\subsetneq\overline{\mathcal{P}}.
\]
In particular, \(\Psi(\mathbb{D}^3)\subsetneq\mathcal{P}\). Moreover, $\Psi(\mathbb D^3)$ is not an open subset of $\mathbb C^3$. Also, 
\[
\widehat{\Psi(\overline{\mathbb{D}}^3)}=\overline{\mathcal{P}} \quad \text{and} \quad \text{int}\left(\widehat{\Psi(\overline{\mathbb{D}}^3)}\right)=\mathcal{P},
\]
where $\widehat{K}$ and $\text{int}(K)$ denote the polynomial convex hull and interior of a set $K$, respectively.
\end{proposition}

\begin{proof}
The identity \(\Psi(\mathbb{T}^3)=b\mathcal{P}\) was proved in Section \ref{sec_02}. Let \((\xi,\lambda_1,\lambda_2)\in\overline{\mathbb{D}}\times\mathbb{D}^2\) and write $(a,s,p)=\Psi(\xi,\lambda_1,\lambda_2)$. Then \((s,p)=(\lambda_1+\lambda_2,\lambda_1\lambda_2)\in\mathbb{G}_2\). The standard membership criterion for $\mathcal P$ from \cite{AglerLykovaYoung2015} reduces the assertion to
\[
|a|
<
\frac12|1-\overline{\lambda}_1\lambda_2|
+
\frac12\sqrt{(1-|\lambda_1|^2)(1-|\lambda_2|^2)}.
\]
Since \(|\xi|\leq1\), it follows that $
2|a|
\leq
|\lambda_1-\lambda_2|
<
|1-\overline{\lambda}_1\lambda_2|$, where the strict inequality is the usual pseudohyperbolic inequality for \(\lambda_1,\lambda_2\in\mathbb{D}\). Hence \(\Psi(\overline{\mathbb{D}}\times\mathbb{D}^2)\subseteq\mathcal{P}\). The inclusion
\(\Psi(\overline{\mathbb{D}}^3)\subseteq\overline{\mathcal{P}}\) follows by continuity. The points of the form $(a,0,0)$, where $0<|a|<1$, belong to $\mathcal P$ but do not belong to $\Psi(\mathbb D^3)$. Moreover, points of the form $(0, 2\lambda, \lambda^2)$ for $|\lambda|<1$ are not interior points of $\Psi(\mathbb D^3)$. Hence, $\Psi(\mathbb D^3)$ is not an open subset of $\mathcal P$. Similarly, for $0<|a|\leq 1$, the points $(a,0,0)$ belong to the closed pentablock $\overline{\mathcal P}$ but do not belong to $\Psi(\overline{\mathbb D}^3)$. Since $\overline{\mathcal P}$ is polynomially convex,
$
\widehat{\Psi(\overline{\mathbb{D}}^3)} \subseteq \widehat{\overline{\mathcal P}}=\overline{\mathcal{P}}.
$
Let $z \in \overline{\mathcal P}$ and let $q \in \mathbb{C}[a, s, p]$. Then
\[
|q(z)| \leq \|q\|_{\infty, \overline{\mathcal P}}=\|q\|_{\infty, b \mathcal P}=\|q\|_{\infty, \Psi(\mathbb T^3)} \leq \|q\|_{\infty, \Psi(\overline{\mathbb D}^3)}.
\]
By definition of the polynomial convex hull, $z \in \widehat{\Psi(\overline{\mathbb{D}}^3)}$ and so, $\overline{\mathcal P} \subseteq \widehat{\Psi(\overline{\mathbb{D}}^3)}$. The desired conclusion now follows. 
\end{proof}

We will use the terminology of \emph{uniqueness sets} in the spirit of \cite{Chee1978}. The following definition is adapted to the class of holomorphic functions on $\mathcal P$.

\begin{definition}
Let $\Omega\subseteq\mathbb C^n$ be a domain and let
$\mathcal F$ be a family of functions on $\Omega$. A set
$E\subseteq\Omega$ is called a \emph{uniqueness set for $\mathcal F$}
if $f\in\mathcal F$ and $f|_E=0$ implies $f\equiv0$ on $\Omega$. 
\end{definition}

\begin{theorem} \label{thm:Psi-uniqueness}
The set $\Psi(\mathbb D^3)$ is a uniqueness set for
$\operatorname{Hol}(\mathcal P)$.
\end{theorem}

\begin{proof}
   Let $f \in \text{Hol}(\mathcal P)$, and let $f=0$ on $\Psi(\mathbb D^3)$. It follows from Proposition \ref{prop:Psi-interior} that $\Psi(\overline{\mathbb{D}}\times\mathbb{D}^2)\subseteq\mathcal{P}$. Since $f$ is holomorphic on $\mathcal P$, one can use the continuity of $f$ on $\mathcal P$ to conclude that $f=0$ on $\Psi(\overline{\mathbb D} \times \mathbb D^2)$. For $0<r<1$, define $f_r=f\circ \delta_r$. Clearly, each $f_r$ is a holomorphic function on $\overline{\mathcal P}$. Fix $r \in (0, 1)$. Let $(a, s, p) \in b \mathcal P$. Then there exists $(\xi, \lambda_1, \lambda_2) \in \mathbb T^3$ such that $\Psi(\xi, \lambda
   _1, \lambda_2)=(a, s, p)$. Since $f|_{\Psi(\overline{\mathbb{D}}\times\mathbb{D}^2)}=0$, we have 
   \[
   f_r(a, s, p)=f(ra, rs, r^2p)=f\left(\Psi(\xi, r\lambda_1, r\lambda_2\right))=0
   \]
and so, $f_r|_{b \mathcal P}=0$. Since $b \mathcal P$ is the distinguished boundary of $\overline{\mathcal P}$, it follows that $f_r=0$ on $\overline{\mathcal{P}}$ for every $r \in (0, 1)$. Therefore, $f=0$ on $\mathcal P_r$ for $0<r<1$, where $\mathcal P_r=\{(ra, rs, r^2p): (a, s, p) \in \overline{\mathcal{P}}\}$. The authors of  \cite{AglerLykovaYoung2015} proved that $\mathcal P=\cup_{0<r<1}\mathcal P_r$ and so, $f=0$ on $\mathcal P$. This completes the proof. 
\end{proof}

The preceding result shows that $\Psi(\mathbb D^3)$ has another useful
property. Although $\Psi(\mathbb D^3)$ need not be open in $\mathcal P$, the
values of a holomorphic function on $\mathcal P$ on this set determine the
function on all of $\mathcal P$. More precisely, if
$f,g\in\operatorname{Hol}(\mathcal P)$ and
\[
f=g \quad \text{on } \Psi(\mathbb D^3),
\]
then $f=g$ on $\mathcal P$. This provides a further justification for the role of $\Psi$ in the construction
of $H^2(\mathcal P)$. The same map that parametrizes the distinguished boundary
through $\Psi(\mathbb T^3)=b\mathcal P$ also produces a subset of $\mathcal P$
on which holomorphic functions are uniquely determined. In this sense, the
choice of $\Psi$ is not merely a convenient parametrization for defining the
Hardy space. It captures, in a natural way, both the boundary structure and the holomorphic uniqueness structure of $\mathcal P$. Our next result shows that $H^2(\mathcal P)$ can be embedded isometrically into the space $H^2(\mathbb D^3)$.

\begin{proposition}\label{prop:isometric-embedding}
Define
\[
J:H^2(\mathcal{P})\longrightarrow H^2(\mathbb{D}^3),
\qquad
Jf=f\circ\Psi.
\]
Then \(J\) is an isometric embedding. 
\end{proposition}

\begin{proof}
Let \(f\in H^2(\mathcal{P})\) and set \(F=Jf\). By Proposition \ref{prop:Psi-interior}, \(F\in\operatorname{Hol}(\mathbb{D}^3)\). For each fixed \(0<r<1\), the function
$
(\xi,\lambda_1,\lambda_2)
\longmapsto
F(\xi,r\lambda_1,r\lambda_2)
$
extends continuously to \(\mathbb{T}^3\), because
\(\Psi(\overline{\mathbb{D}}\times r\overline{\mathbb{D}}^2)\Subset\mathcal{P}\). We use the following equivalent form of the Hardy norm on \(\mathbb{D}^3\):
\begin{equation}\label{mixed-Hardy-norm}
\|F\|_{H^2(\mathbb{D}^3)}^2
=
\sup_{0<r<1}
\int_{\mathbb{T}^3}
|F(\xi,r\lambda_1,r\lambda_2)|^2\,\mathrm{d}m_3.
\end{equation}
Indeed, if
\[
F(\xi, \lambda_1, \lambda_2)
=
\sum_{m,n,k\geq0}c_{m,n,k}\xi^m \lambda_1^n \lambda_2^k,
\]
then Parseval's identity shows that the integral on the right-hand side of \eqref{mixed-Hardy-norm} equals
\[
\sum_{m,n,k\geq0}|c_{m,n,k}|^2r^{2(n+k)},
\]
and its supremum is \(\sum_{m,n,k\geq0}|c_{m,n,k}|^2\), the usual \(H^2(\mathbb{D}^3)\)-norm squared. Using the definition of \(\mu=\Psi_*m_3\) and Theorem \ref{thm:radial-realization}, we obtain
\[
\begin{aligned}
\|Jf\|_{H^2(\mathbb{D}^3)}^2
=
\sup_{0<r<1}
\int_{\mathbb{T}^3}
|f\circ\Psi(\xi, r\lambda_1, r\lambda_2)|^2\,\mathrm{d}m_3
&=
\sup_{0<r<1}
\int_{b\mathcal{P}}|f\circ\delta_r(a,s,p)|^2\,\mathrm{d}\mu(a,s,p)\\
&=
\|f\|_{H^2(\mathcal{P})}^2.
\end{aligned}
\]
Thus \(J\) is an isometry and its range is closed.
\end{proof}

It was proved in Theorem \ref{thm:hol-realization} that $H^2(\mathcal P)$ is a reproducing kernel space. Since the map $J: H^2(\mathcal P) \to H^2(\mathbb D^3)$ is an isometry, it follows that $J(H^2(\mathcal P))$ itself is a reproducing kernel Hilbert space with the reproducing kernel described below.

\begin{proposition}
Let $J\colon H^2(\mathcal{P}) \to H^2(\mathbb{D}^3)$ be the isometric embedding defined by $J(f) = f \circ \Psi$, and set $\mathcal{M} = J(H^2(\mathcal{P}))$. Then the reproducing kernel of $\mathcal{M}$ is given by
\[
K_{\mathcal{M}}(z,w) = \bigl(JJ^*S_{\mathbb{D}^3}(\cdot,w)\bigr)(z), \qquad z, w\in\mathbb{D}^3,
\]
where $S_{\mathbb{D}^3}$ denotes the Szeg\"{o} kernel of $H^2(\mathbb{D}^3)$.
\end{proposition}

\begin{proof}
Since $J$ is an isometric embedding, it is a unitary operator from $H^2(\mathcal{P})$ onto its range $\mathcal{M}$. The orthogonal projection onto $\mathcal{M}$ is thus $P_{\mathcal{M}} = JJ^*$. Fix $w\in\mathbb{D}^3$. For any $F\in\mathcal{M}$, we have that
\[
F(w) = \left\langle F, S_{\mathbb{D}^3}(\cdot,w) \right\rangle_{H^2(\mathbb{D}^3)} = \left\langle F, P_{\mathcal{M}}S_{\mathbb{D}^3}(\cdot,w) \right\rangle_{H^2(\mathbb{D}^3)}.
\]
Hence $K_{\mathcal{M}}(\cdot,w) = P_{\mathcal{M}}S_{\mathbb{D}^3}(\cdot,w)$ belongs to $\mathcal{M}$ and satisfies the reproducing property, so it is the reproducing kernel of $\mathcal{M}$.
\end{proof}

The reproducing kernel of $H^2(\mathcal P)$ also reflects the boundary
geometry of the pentablock. In view of the orthogonal decomposition
\[
H^2(\mathcal P)=\bigoplus_{n=0}^{\infty}\mathcal Q_n,
\]
the dimensions of the weighted homogeneous components can be used to
estimate the growth of the reproducing kernel under the radial
dilation maps
\[
\delta_r(a,s,p)=(ra,rs,r^2p),\qquad 0<r<1.
\]
We next obtain a quantitative lower bound for the diagonal of the
reproducing kernel along these dilated boundary points. In particular,
this shows that the norms of the point-evaluation functionals become
unbounded as the boundary is approached.

\begin{proposition}\label{prop:kernel-growth}
Let $K_{\mathcal P}$ be the reproducing kernel of $H^2(\mathcal P)$ and let $r \in (0, 1)$. Then
\[
\int_{b\mathcal P}
K_{\mathcal P}(\delta_r z,\delta_r z)\,d\mu(z)
=
\frac{1}{(1-r^2)^3(1+r^2)}
\ \ \text{and} \ \ \sup_{z\in b\mathcal P}
K_{\mathcal P}(\delta_r z,\delta_r z)
\geq
\frac{1}{(1-r^2)^3(1+r^2)}.
\]
In particular, $\displaystyle \sup_{z\in b\mathcal P}K_\mathcal{P}(\delta_r z, \delta_r z)
\to \infty$ as
$r\uparrow1$.
\end{proposition}

\begin{proof}
We have by Proposition \ref{prop:decompositions} that $H^2(\mathcal P)=\bigoplus_{n=0}^{\infty}\mathcal Q_n$. Set
$d_n =\dim\mathcal Q_n$. For each $n\geq0$, let
$\{e_{n,k}:1\leq k\leq d_n\}$ be an orthonormal basis for
$\mathcal Q_n$. The reproducing kernel for $\mathcal{Q}_n$ is
\[
K_{\mathcal{Q}_n}(z,w) = \sum_{k=1}^{d_n} e_{n,k}(z)\,\overline{e_{n,k}(w)} \quad \text{and so,} \quad 
K_{\mathcal{P}}(z,w) = \sum_{n=0}^{\infty} \sum_{k=1}^{d_n} e_{n,k}(z)\,\overline{e_{n,k}(w)}.
\] 
Since $e_{n,k}\in\mathcal Q_n$ is weighted homogeneous of degree $n$, we have that $
e_{n,k}(\delta_r z)=r^n e_{n,k}(z)$. Then
\[
\begin{aligned}
K_{\mathcal P}(\delta_r z,\delta_r z)
&=
\sum_{n=0}^{\infty}
\sum_{k=1}^{d_n}
\left|e_{n,k}(\delta_r z)\right|^2
=
\sum_{n=0}^{\infty}
r^{2n}
\sum_{k=1}^{d_n}|e_{n,k}(z)|^2.
\end{aligned}
\]
Integrating over $b\mathcal P$ and using Tonelli's theorem, we have
\[
\begin{aligned}
\int_{b\mathcal P}
K_{\mathcal P}(\delta_r z,\delta_r z)\,d\mu(z)
=
\sum_{n=0}^{\infty}
r^{2n}
\sum_{k=1}^{d_n}
\int_{b\mathcal P}|e_{n,k}(z)|^2\,d\mu(z)
&=
\sum_{n=0}^{\infty}r^{2n}d_n\\
&=\frac{1}{(1-r^2)^3(1+r^2)},
\end{aligned}
\]
where in the last equality we used the fact that
$d_n=\left\lfloor\frac{(n+2)^2}{4}\right\rfloor$. Since $\mu$ is a probability measure, it follows that
\[
\frac{1}{(1-r^2)^3(1+r^2)}=\int_{b\mathcal P}
K_{\mathcal P}(\delta_r z,\delta_r z)\,d\mu(z)
\leq
\sup_{z\in b\mathcal P}
K_{\mathcal P}(\delta_r z,\delta_r z).
\]
This completes the proof.
\end{proof}

Building on the integral representation of the reproducing kernel $K_{\mathcal{P}}$ and the norm estimates for evaluation functionals obtained above, we now introduce the multiplier algebra of $H^2(\mathcal{P})$, which characterizes all bounded multiplication operators on the space.

\begin{definition}
A holomorphic function $\varphi$ on $\mathcal{P}$ is a \textit{multiplier} of $H^2(\mathcal{P})$ if $\varphi f\in H^2(\mathcal{P})$ for all $f\in H^2(\mathcal{P})$. The set of all multipliers is denoted $\operatorname{Mult}(H^2(\mathcal{P}))$, with equivalent characterization
\[
\operatorname{Mult}(H^2(\mathcal{P})) = \{\varphi\in\operatorname{Hol}(\mathcal{P}) : M_\varphi\in\mathcal{B}(H^2(\mathcal{P}))\},
\]
where $M_\varphi$ is the multiplication operator. Equipped with the norm $\|\varphi\|_{\operatorname{Mult}} = \|M_\varphi\|$, $\operatorname{Mult}(H^2(\mathcal{P}))$ is the \emph{multiplier algebra} of $H^2(\mathcal{P})$. Recall $H^\infty(\mathcal{P})$ denotes the space of bounded holomorphic functions on $\mathcal{P}$, equipped with the supremum norm $\|\cdot\|_{\infty,\mathcal{P}}$.
\end{definition}

\begin{theorem}\label{thm:multipliers}
The multiplier algebra of $H^2(\mathcal{P})$ is isometrically isomorphic to $H^\infty(\mathcal{P})$.
\end{theorem}

\begin{proof}
Define $\Psi_r=\delta_r \circ \Psi$ for $0<r<1$. First suppose $\varphi\in H^\infty(\mathcal{P})$. For any $f\in H^2(\mathcal{P})$,
\begin{align*}
\|\varphi f\|_{H^2(\mathcal{P})}^2
&= \sup_{0<r<1} \int_{\mathbb{T}^3} |\varphi(\Psi_r(\zeta))|^2 |f(\Psi_r(\zeta))|^2\,\mathrm{d}m_3(\zeta) \\
&\le \|\varphi\|_{\infty,\mathcal{P}}^2 \sup_{0<r<1} \int_{\mathbb{T}^3} |f(\Psi_r(\zeta))|^2\,\mathrm{d}m_3(\zeta) \\
&= \|\varphi\|_{\infty,\mathcal{P}}^2 \|f\|_{H^2(\mathcal{P})}^2.
\end{align*}
Thus $M_\varphi$ is bounded with $\|M_\varphi\| \le \|\varphi\|_{\infty,\mathcal{P}}$. Conversely, let $\varphi\in\operatorname{Mult}(H^2(\mathcal{P}))$. For any $\lambda\in\mathcal{P}$, the reproducing property gives $M_\varphi^*k_\lambda = \overline{\varphi(\lambda)}k_\lambda$, so
\[
|\varphi(\lambda)| = \frac{\|M_\varphi^*k_\lambda\|}{\|k_\lambda\|} \le \|M_\varphi\|.
\]
Hence $\varphi\in H^\infty(\mathcal{P})$ with $\|\varphi\|_{\infty,\mathcal{P}} \le \|M_\varphi\|$. Combining the two inequalities yields $\|M_\varphi\| = \|\varphi\|_{\infty,\mathcal{P}}$, completing the proof.
\end{proof}

\begin{corollary}
The map $\pi_\infty \colon H^\infty(\mathcal{P}) \to \mathcal{B}(H^2(\mathcal{P})), \pi_\infty(\varphi) = M_\varphi$, is an isometric algebra isomorphism onto $\operatorname{Mult}(H^2(\mathcal{P}))$. Consequently, $H^\infty(\mathcal{P})$ is a norm-closed unital operator algebra on $H^2(\mathcal{P})$.
\end{corollary}

\begin{proof}
By Theorem~\ref{thm:multipliers}, $\pi_\infty$ is an isometric algebra isomorphism onto its range. Since $H^\infty(\mathcal{P})$ is complete under the supremum norm and $\|M_\varphi\| = \|\varphi\|_{\infty,\mathcal{P}}$, the range of $\pi_\infty$ is complete, hence a norm-closed subalgebra of $\mathcal{B}(H^2(\mathcal{P}))$.
\end{proof}

Next, we present the following boundary realization theorem for $H^\infty(\mathcal P)$. 

\begin{theorem}
Let $f \in H^\infty(\mathcal P)$ and let $f^* \in L^2(b \mathcal P, \mu)$ be the unique function defined as in Corollary \ref{radial}. Then $f^* \in L^\infty(b \mathcal P, \mu)$ and $  \|f^*\|_{\infty, b \mathcal P}=\|f\|_{\infty,\mathcal P}$. In particular, the boundary value map
    \[
    \tau_{\infty}: H^\infty(\mathcal{P})\rightarrow L^\infty(b\mathcal{P},\mu), \quad f \mapsto f^*
    \]
    is an isometric algebra embedding.  
\end{theorem}

\begin{proof}
Let $f \in H^\infty(\mathcal P)$, and let $f_r: \overline{\mathcal P} \to \mathbb C$ be defined as $f_r(a, s, p)=f(ra, rs, r^2p)$ for $0<r<1$.  Since $\lVert f_r\rVert_{L^\infty(b\mathcal{P},\mu)} \le \lVert f\rVert_{\infty,\mathcal{P}}$, the sequence $\{f_r\}$ is uniformly bounded in $L^\infty(b\mathcal{P},\mu)$. By Banach Alaoglu theorem, there exists a subsequence $\{f_{r_k}\}$ converging to some $g \in L^\infty(b\mathcal P, \mu)$ in the weak-$*$ topology, that is, for every $h \in L^1(b\mathcal P, \mu)$,
\[
\int_{b\mathcal P}f_{r_k}h d\mu \to \int_{b\mathcal P} gh d\mu 
\]
as $r_k \uparrow 1$. Since $H^\infty (\mathcal P) \subseteq H^2(\mathcal P)$, we have by Corollary \ref{radial} that there exists a unique boundary function $f^* \in L^2(b\mathcal P, \mu)$ such that $f_r \to f^*$ in $L^2(b \mathcal P, \mu)$ and in particular, $f_{r_k} \to f^*$ in $L^2(b \mathcal P, \mu)$. Next, we show that $f_{r_k} \xrightarrow{w^*} f^*$ in $L^\infty (b \mathcal P, \mu)$, that is, for every $h \in L^1(b \mathcal P, \mu)$,
\[
\int_{b\mathcal P}f_{r_k}h d\mu \to \int_{b\mathcal P} f^*h d\mu. 
\]
Note that $f^* \in L^\infty (b \mathcal P, \mu)$, because $f \in H^\infty (\mathcal P)$, $f_r \to f^*$ in $L^2(b \mathcal P, \mu)$ and hence a subsequence converges to $f^*$ almost everywhere. Next, take $h \in L^1(b \mathcal P, \mu)$ and let $\varepsilon>0$ be given. Since $L^2(b\mathcal P, \mu)$ is dense in $L^1(b \mathcal P, \mu)$, there exists $\varphi \in L^2(b \mathcal P, \mu)$ such that $\|h-\varphi\|_{L^1(b \mathcal P, \mu)}< \varepsilon \slash (2\|f\|_{\infty,  \mathcal P}+2\|f^*\|_{\infty, b \mathcal P})$. Also, one can choose $N \in \mathbb N$ such that for every $r_k \geq N$, we have $\|f_{r_k}-f^*\|_{L^2(b \mathcal P, \mu)} \leq  \varepsilon \slash (2 \|\varphi\|_{L^2(b \mathcal P, \mu})$. Then
\begin{align*}
\left|\int_{b\mathcal P}(f_{r_k}-f^*)h d\mu \right| 
& \leq \left|\int_{b\mathcal P}(f_{r_k}-f^*)\varphi d\mu \right|+\left|\int_{b\mathcal P}(f_{r_k}-f^*)(h-\varphi) d\mu \right| \\
& \leq \|f_{r_k}-f^*\|_{L^2} \|\varphi\|_{L^2} + (\|f\|_{\infty, \mathcal P}+\|f^*\|_{\infty, b \mathcal P}) \|h-\varphi\|_{L^1}\\
& \leq \varepsilon
\end{align*}
for every $r_k \geq N$.  Thus, $f_{r_k}$ converges to $f^*$ in the weak-$*$ topology of $L^\infty(b \mathcal P, \mu)$. By uniqueness of the weak-$*$ limit, it follows that $f^*=g$ in $L^\infty(b \mathcal P, \mu)$. Then 
\[
\|f^*\|_{\infty, b \mathcal P} \leq \underset{r_k}{\lim \inf} \|f_{r_k}\|_{\infty, b \mathcal P} \leq \|f\|_{\infty, \mathcal P} \quad \text{and so,} \quad \|f^*\|_{\infty, b \mathcal P} \leq \|f\|_{\infty, \mathcal P}. 
\]  
Next, we prove the reverse inequality. Let $h \in H^2(\mathcal P)$. We claim that $(fh)^*=f^*h^*$. Observe that $(fh)_r=f_rh_r$ for $0<r<1$. Since $\mu(b \mathcal P)=1$, we have that $\|.\|_{L^1} \leq \|.\|_{L^2}$. Then by Corollary \ref{radial},
\begin{align*}
    \|f_rh_r-f^*h^*\|_{L^1} 
    & \leq \|f_r(h_r-h^*)\|_{L^1}+\|(f_r-f^*)h^*\|_{L^1} \\
    & \leq \|f_r\|_{L^\infty}\|h_r-h^*\|_{L^1}+\|f_r-f^*\|_{L^2}\|h^*\|_{L^2}\\
    & \leq \|f_r\|_{L^\infty}\|h_r-h^*\|_{L^2}+\|f_r-f^*\|_{L^2}\|h^*\|_{L^2}\\
    & \leq \|f\|_{\infty, \mathcal P}\|h_r-h^*\|_{L^2}+\|f_r-f^*\|_{L^2}\|h^*\|_{L^2}\\
    & \to 0 \quad (r \uparrow 1)
\end{align*}
and  $\|f_rh_r-(fh)^*\|_{L^1}=\|(fh)_r-(fh)^*\|_{L^1} \leq \|(fh)_r-(fh)^*\|_{L^2} \to 0$ as $r \uparrow 1$. By uniqueness of the limits in $L^1(b \mathcal P, \mu)$, $(fh)^*=f^*h^*$. This also implies that $\tau_\infty(h_1h_2)=\tau_\infty(h_1)\tau_\infty(h_2)$ for every $h_1, h_2 \in H^\infty(\mathcal P)$. The linearity of $\tau_\infty$ follows trivially and thus $\tau_\infty$ is an algebra morphism. Again by Corollary \ref{radial}, 
\[
\|M_fh\|=\|fh\|_{H^2(\mathcal P)}=\|(fh)^*\|_{L^2}\leq \|f^*h^*\|_{L^2}\leq \|f^*\|_{\infty, b \mathcal P}\|h^*\|_{L^2}=\|f^*\|_{\infty, b \mathcal P}\|h\|_{H^2(\mathcal P)}.
\]
By Theorem \ref{thm:multipliers}, $\|f\|_{\infty, \mathcal P}=\|M_f\| \leq \|f^*\|_{\infty, b \mathcal P}$. Therefore, $\|\tau_\infty(f)\|_{L^\infty(b \mathcal P, \mu)}=\|f\|_{H^2(\mathcal P)}$ showing that $\tau_\infty$ is an isometry. The proof is now complete. 
\end{proof}

Putting everything together, we have the following commutative diagram: 

\begin{center}
\begin{tikzpicture}[>=stealth]

    \node (Hinf) at (0,4) {$H^\infty(\mathcal P)$};
    \node (Linf) at (7,4) {$L^\infty(b\mathcal P,\mu)$};

    \node (H2P) at (0,2) {$H^2(\mathcal P)$};
    \node (L2P) at (7,2) {$L^2(b\mathcal P,\mu)$};

    \node (H2D) at (0,0) {$H^2(\mathbb D^3)$};
    \node (L2T) at (7,0) {$L^2(\mathbb T^3, m)$};

    \draw[->]
        (Hinf) -- node[above] {$f\mapsto f^*$} (Linf);

    \draw[->]
        (H2P) -- node[above] {$f\mapsto f^*$} (L2P);

    \draw[->]
        (H2D) -- node[above] {\small{radial boundary map}} (L2T);

    \draw[->, right=20]
        (Hinf) to
        node[left] {$f\mapsto f$} (H2P);

    \draw[->]
        (H2P) -- node[left] {$f\mapsto f\circ\Psi$} (H2D);

    \draw[->, left=20]
        (Linf) to
        node[right] {$g\mapsto g$} (L2P);

    \draw[->]
        (L2P) -- node[right] {$g\mapsto g\circ\Psi$} (L2T);

\end{tikzpicture}
\end{center}

\section{Brown-Halmos characterization and commutant lifting for $\mathcal P$}
\label{sec_04}

\noindent In this section, we study Toeplitz operators on $H^2(\mathcal P)$ and the coordinate multiplication operators $(T_a,T_s,T_p)$ on $H^2(\mathcal P)$. We first establish the Brown-Halmos identities that characterizes the Toeplitz operators. Furthermore,  we  study the commutant of the coordinate multiplication operators and the joint reducing subspaces. Finally, we establish a commutant lifting theorem for $\mathcal P$-isometries. Let 
 \[\
P_{H^2(\mathcal P)}: L^2(b \mathcal P, \mu) \to H^2(\mathcal P)
 \]
denote the orthogonal projection. For $\varphi \in L^\infty(b \mathcal P, \mu)$, let $M_{\varphi}: L^2(b \mathcal P, \mu) \to L^2(b\mathcal P, \mu)$ be the multiplication operator defined by 
\[
M_{\varphi}(f)(a, s, p)=\varphi(a, s, p)f(a, s, p)
\]
for every $f \in L^2(b \mathcal P, \mu)$. Evidently, $M_\varphi$ is a bounded normal operator such that $\|M_\varphi\|=\|\phi\|_{\infty, b \mathcal P}$ and $M_\varphi^*=M_{\overline{\varphi}}$. For $\varphi \in L^\infty(b \mathcal P, \mu)$, the Toeplitz operator with symbol $\varphi$ is the bounded operator $T_\varphi: H^2(\mathcal P) \to H^2(\mathcal P)$ given by
\[
T_\varphi=P_{H^2(\mathcal P)}M_\varphi|_{H^2(\mathcal P)}.
\]
The coordinate functions $a, s, p$ are holomorphic polynomials on $\mathcal P$. Since the polynomial algebra is invariant under multiplication by each of the coordinate functions and is dense in $H^2(\mathcal P)$, it follows that 
\[
M_a(H^2(\mathcal P)) \subseteq H^2(\mathcal P), \quad M_s(H^2(\mathcal P) \subseteq H^2(\mathcal P) \quad \text{and} \quad M_p(H^2(\mathcal P)) \subseteq H^2(\mathcal P).
\]
Consequently, the corresponding Toeplitz operators $T_a, T_s$ and $T_p$ are simply the restrictions of the multiplication operators $M_a, M_s$ and $M_p$, respectively. The multiplication operators associated with $H^2(\mathcal P)$ and $L^2(b\mathcal P,\mu)$ give natural examples of isometries and unitaries, respectively, associated with $\mathcal P$. This naturally lead to the definitions of $\mathcal P$-isometries and $\mathcal P$-unitaries, which we recall next from \cite{PalTomar2026}.

\begin{definition}
    A commuting triple $(A, S, P)$ of operators acting on a Hilbert space $\mathcal H$ is said to be 
    \begin{enumerate}[leftmargin=*]
        \item a \textit{$\mathcal P$-unitary} if its joint spectrum $\sigma_T(A, S, P) \subseteq b \mathcal P$;
        \item a \textit{$\mathcal P$-isometry} if it is the restriction of a $\mathcal P$-unitary to a joint invariant subspace.
    \end{enumerate}
\end{definition}

\begin{lemma}
    The triple $(M_a, M_s, M_p)$ is a $\mathcal P$-unitary with $\sigma_T(M_a, M_s, M_p)=b\mathcal P$. Also, $(T_a, T_s, T_p)$ is a $\mathcal P$-isometry such that $T_p$ is a pure isometry and $\sigma_T(T_a, T_s, T_p)=\overline{\mathcal P}$. Moreover, $(M_a, M_s, M_p)$ acting on $L^2(b \mathcal P, \mu)$ is the minimal normal extension of $(T_a, T_s, T_p)$ acting on $H^2(\mathcal P)$.
\end{lemma}

\begin{proof}
    Since $(M_a, M_s, M_p)$ is a commuting triple on $L^2(b \mathcal P, \mu)$, its Taylor joint spectrum coincides with the joint essential range of the coordinate functions $(a, s, p)$. As $\mu$ is supported on $b \mathcal P$, we have that $\sigma_T(M_a, M_s, M_p) \subseteq b\mathcal{P}$. Hence, $(M_a, M_s, M_p)$ is a $\mathcal P$-unitary.  Next, we show that $\sigma_T(M_a, M_s, M_p)=b\mathcal P$. Take
$x=(a_0,s_0,p_0)\in b\mathcal P$.
Let $U$ be any open neighbourhood of $x$ in $\mathbb C^3$. Since the
coordinate map $(a, s, p)$ is identity on $b\mathcal P$, we have that $(a,s,p)^{-1}(U)=U\cap b\mathcal P$. Thus, it suffices to show that
$\mu(U\cap b\mathcal P)>0$. Since
$\mu=\Psi_*m_3$, we have that
$
\mu(U\cap b\mathcal P)
=
m\bigl(\Psi^{-1}(U\cap b\mathcal P)\bigr)$. As $\Psi:\mathbb T^3\to b\mathcal P$ is continuous,
$\Psi^{-1}(U\cap b\mathcal P)$ is an open subset of $\mathbb T^3$.
Also, $\Psi$ is surjective and
$U\cap b\mathcal P\neq\varnothing$, so
$\Psi^{-1}(U\cap b\mathcal P)$ is nonempty. Every nonempty open subset of $\mathbb T^3$ contains a basic open set
of the form
$
I_1\times I_2\times I_3,
$
where each $I_j$ is a nonempty open arc of $\mathbb T$. Since the
normalized Haar measure on $\mathbb T^3$ is the product of the
normalized arc-length measures on the three copies of $\mathbb T$, we
have that
$
m_3(I_1\times I_2\times I_3)
=
m_1(I_1)m_1(I_2)m_1(I_3)
>
0.
$
Hence, $
m_3\bigl(\Psi^{-1}(U\cap b\mathcal P)\bigr)>0$ and thus $\mu(U\cap b\mathcal P)>0$. Thus, $\mu\bigl((a,s,p)^{-1}(U)\bigr)>0$ for every open neighbourhood $U$ of $x$, showing that
$x\in\operatorname{ess\,ran}(a,s,p)$. Thus, $\sigma_T(M_a,M_s,M_p)=b\mathcal P$.  

\smallskip 

Since $(T_a, T_s, T_p)$ is the restriction of $(M_a, M_s, M_p)$ to a joint invariant subspace, the triple $(T_a, T_s, T_p)$ is a $\mathcal P$-isometry and thus $\sigma_T(T_a, T_s, T_p) \subseteq \overline{\mathcal P}$. Since the last component of a $\mathcal P$-isometry is an isometry, $T_p$ is an isometry. As $H^2(\mathcal P)$ is a reproducing kernel Hilbert space, for every
$x=(a,s,p)\in\mathcal P$, there exists a reproducing kernel
$k_x\in H^2(\mathcal P)$ satisfying $f(x)=\langle f,k_x\rangle$  for every $f\in H^2(\mathcal P)$. Then
$
\langle f,T_a^*k_x\rangle
=\langle T_af,k_x\rangle=(af)(x)=a\,f(x)=\langle f,\overline{a}\,k_x\rangle$ for every $f \in H^2(\mathcal P)$.
Hence, $T_a^*k_x=\overline{a}\,k_x$.
Similarly,  $T_s^*k_x=\overline{s}\,k_x$ and $T_p^*k_x=\overline{p}\,k_x$. Thus, $k_x$ is a joint eigenvector of the commuting triple
$(T_a^*,T_s^*,T_p^*)$ corresponding to the joint eigenvalue
$(\overline{a},\overline{s},\overline{p})$. Therefore, $\overline{\mathcal P} \subseteq \sigma_T(T_a^*,T_s^*,T_p^*)$ and so, $\overline{\mathcal P} \subseteq \sigma_T(T_a, T_s, T_p)$. Therefore, $\sigma_T(T_a, T_s, T_p)=\overline{\mathcal P}$. Next, we show that $T_p$ is pure. For $k \ge 0$, define
\[
H_{\ge k} = \bigoplus_{n \ge k} \mathcal{Q}_n.
\]
By the $2$-degree homogeneity of $T_p$, we have $T_p^n(H^2(\mathcal{P})) \subset H_{\ge 2n}$ for all $n \ge 0$. For any $f \in H^2(\mathcal{P})$, it follows that
$
\lVert T_p^{*n} f\rVert = \lVert P_{\operatorname{Ran} T_p^n} f\rVert \le \lVert P_{H_{\ge 2n}} f\rVert \to 0$ as $n \to \infty$ and thus $T_p$ is pure. It remains to verify the minimality. Let
\[\mathcal M=
\overline{\operatorname{span}}
\left\{
M_a^{*m}M_s^{*n}M_p^{*k}f:
f\in H^2(\mathcal P),
m, n, k \geq 0
\right\}.
\] 
Then 
$
\mathcal M=
\overline{\operatorname{span}}
\left\{
M_a^{*m}M_s^{*n}M_p^{*k}q:
q\in\mathbb C[a,s,p],\;
m,n,k\ge0
\right\}
$
since the polynomials are dense in $H^2(\mathcal P)$. Note that $M_a^{*m}M_s^{*n}M_p^{*k}q
=
\overline a^{\,m}\overline s^{\,n}\overline p^{\,k}\,q$. Therefore,
$
\mathcal M
=
\overline{\mathcal A}^{\,L^2(b\mathcal P,\mu)}$, where
$
\mathcal A
=
\operatorname{span}
\left\{
a^is^jp^\ell
\overline a^{\,m}
\overline s^{\,n}
\overline p^{\,k}
:
i,j,\ell,m,n,k\ge0
\right\}.
$
The algebra $\mathcal A$ contains the constants, is closed under complex
conjugation and separates the points of $b\mathcal P$. Hence, by the
complex Stone-Weierstrass theorem,
$
\overline{\mathcal A}^{\,\|\cdot\|_\infty}
=
C(b\mathcal P)$. Since $\mu$ is a finite regular Borel measure on 
$b\mathcal P$, the space $C(b\mathcal P)$ is dense in
$L^2(b\mathcal P,\mu)$. Thus,
$
\mathcal M=\overline{\mathcal A}^{\,L^2(b\mathcal P,\mu)}
=
L^2(b\mathcal P,\mu)$. The proof is now complete.
\end{proof}

If $M$ is a bounded operator on $L^2(\mathbb T)$ belonging to the commutant of the operator $M_z$ on $L^2(\mathbb T)$, then there exists a map $\varphi \in L^\infty(\mathbb T)$ such that $M=M_\varphi$. Analogous results have been proved for the symmetrized bidisc \cite{BhattacharyyaDasSau2021} and tetrablock \cite{JainKumarMalPramanick2026} (also see \cite{KumarMalPramanick2025}). Our next result is an analogue of this property for the pentablock. 

\begin{lemma}\label{lem_402}
    Let $M$ be a bounded operator on $L^2(b \mathcal P, \mu)$. If $M$ commutes with $M_a, M_s$ and $M_p$, then there exists $\varphi \in L^\infty(b \mathcal P, \mu)$ such that $M=M_{\varphi}$.
\end{lemma}

\begin{proof}
    Since $(M_a, M_s, M_p)$ is a commuting triple of normal operators with its joint spectrum $\sigma_T(M_a, M_s, M_p)=b \mathcal{P}$, it follows from the spectral theorem that the von Neumann algebra $W^*(M_a, M_s, M_p)$ generated by $M_a, M_s, M_p$ is the maximal abelian von Neumann algebra that satisfies $W^*(M_a, M_s, M_p)=W^*(M_a, M_s, M_p)'=\{M_\varphi: \varphi \in L^\infty(b \mathcal P, \mu)\}$.  The proof is now complete.  
\end{proof}

Several characterizations of $\mathcal P$-unitaries were obtained in \cite{PalTomar2026, JindalKumar2024}. In particular, we have by Theorem 4.2 in \cite{PalTomar2026} that the $\mathcal P$-unitary $(M_a, M_s, M_p)$ satisfies 
\begin{align}\label{eqn_601}
M_a^*M_a+\frac{1}{4}M_s^*M_s=I, \quad M_s=M_s^*M_p \quad \text{and} \quad M_p^*M_p=I.
\end{align}
Let $T = T_\varphi$ be the Toeplitz operator with symbol $\varphi \in L^\infty(b\mathcal{P},\mu)$. For $f, g \in H^2(\mathcal{P})$, we have
\begin{align*}
\langle T_p^* T_\varphi T_p f, g \rangle
= \langle T_\varphi T_p f, T_p g \rangle
= \langle P_{H^2(\mathcal{P})} M_\varphi T_p f, T_p g \rangle
&= \langle M_\varphi M_p f, M_p g \rangle \\
&= \langle M_p M_\varphi f, M_p g \rangle \\
&= \langle M_\varphi f, g \rangle \qquad [\text{by \eqref{eqn_601}}] \\
&= \langle P_{H^2(\mathcal{P})} M_\varphi f, g \rangle \\
&= \langle T_\varphi f, g \rangle,
\end{align*}
\quad\text{and}\quad
\begin{align*}
\langle T_s^* T_\varphi T_p f, g \rangle
= \langle T_\varphi T_p f, T_s g \rangle
= \langle P_{H^2(\mathcal{P})} M_\varphi T_p f, T_s g \rangle
&= \langle M_\varphi M_p f, M_s g \rangle \\
&= \langle M_p M_\varphi f, M_s g \rangle \\
&= \langle M_s^* M_p M_\varphi f, g \rangle \\
&= \langle M_s M_\varphi f, g \rangle \qquad [\text{by \eqref{eqn_601}}] \\
&= \langle M_\varphi M_s f, g \rangle \\
&= \langle P_{H^2(\mathcal{P})} M_\varphi M_s f, g \rangle \\
&= \langle T_\varphi T_s f, g \rangle,
\end{align*}
\quad\text{and}\quad
\begin{align*}
\langle T_a^* T_\varphi T_a f, g \rangle
= \langle T_\varphi T_a f, T_a g \rangle
&= \langle P_{H^2(\mathcal{P})} M_\varphi T_a f, T_a g \rangle \\
&= \langle M_\varphi M_a f, M_a g \rangle \\
&= \langle M_a M_\varphi f, M_a g \rangle \\
&= \langle M_a^* M_a M_\varphi f, g \rangle \\
&= \langle M_\varphi f, g \rangle - \frac{1}{4} \langle M_s^* M_s M_\varphi f, g \rangle \qquad [\text{by \eqref{eqn_601}}] \\
&= \langle M_\varphi f, g \rangle - \frac{1}{4} \langle M_\varphi M_s f, M_s g \rangle \\
&= \langle P_{H^2(\mathcal{P})} M_\varphi f, g \rangle - \frac{1}{4} \langle P_{H^2(\mathcal{P})} M_\varphi M_s f, M_s g \rangle \\
&= \langle T_\varphi f, g \rangle - \frac{1}{4} \langle T_\varphi T_s f, T_s g \rangle \\
&= \left\langle \left( T_\varphi - \frac{1}{4} T_s^* T_\varphi T_s \right) f, g \right\rangle .
\end{align*}
Therefore, 
\[
T_p^* T_\varphi T_p = T_\varphi, \quad T_s^* T_\varphi T_p = T_\varphi T_s \quad \text{and} \quad T_a^* T_\varphi T_a + \frac{1}{4} T_s^* T_\varphi T_s = T_\varphi.
\]
The above identities are the natural Brown-Halmos type relations for Toeplitz operators on $H^2(\mathcal{P})$. This motivates the following definition.
\begin{definition}
    A bounded operator $T$ acting on $H^2(\mathcal P)$ is said to satisfy the Brown-Halmos identities if the following holds:
  \[
T_p^* T T_p = T, \quad T_s^* T T_p = T T_s \quad \text{and} \quad T_a^* T T_a + \frac{1}{4} T_s^* T T_s = T.
\]
\end{definition}

We now present a commutant lifting lemma for operators acting on $H^2(\mathcal P)$ that satisfy the Brown-Halmos identities. The analogous statements appear in earlier works including the symmetrized bidisc \cite{BhattacharyyaDasSau2021}, the tetrablock \cite{JainKumarMalPramanick2026} and the symmetrized polydisc \cite{DasSau2021}. However, the same techniques in the corresponding proofs need not work due to the presence of the third identity in the Brown-Halmos identities for $H^2(\mathcal P)$.

\begin{lemma}\label{lem:lifting}
If $T\in \mathcal B(H^2(\mathcal P))$ satisfies the Brown-Halmos identities, then there exists an operator $M\in \mathcal B(L^2(b\mathcal P,\mu))$ such that $\|M\|=\|T\|$, $M$ commutes with the coordinate multiplication operators $M_a, M_s, M_p$ and $T=P_{H^2(\mathcal P)}M|_{H^2(\mathcal P)}$.
\end{lemma}

\begin{proof}
It follows from the Brown-Halmos identities for $T$ and \eqref{eqn_601} that
\[
\frac12T_a^*TT_a+\frac18T_s^*TT_s+\frac12T_p^*TT_p=T \quad \text{and} \quad 
\frac12M_a^*M_a+\frac18M_s^*M_s+\frac12M_p^*M_p=I.
\]
Consider the commuting normal triple of operators acting on $L^2(b \mathcal P, \mu)$ given by
\[
(N_1, N_2, N_3)=\left(\frac1{\sqrt2}M_a, \frac1{2\sqrt2}M_s, \frac1{\sqrt2}M_p \right).
\]
Evidently, $N_1^*N_1+N_2^*N_2+N_3^*N_3=I$. Define a linear map
\[
\Phi:\mathcal B(L^2(b\mathcal P,\mu))
\to
\mathcal B(L^2(b\mathcal P,\mu)) \quad \text{as} \quad 
\Phi(R)=N_1^*RN_1+N_2^*RN_2+N_3^*RN_3.
\]
Since $N_1^*N_1+N_2^*N_2+N_3^*N_3=I$, it follows that $\|\Phi(R)\|
\leq \|R\|$ for every $R\in\mathcal B(L^2(b\mathcal P,\mu))$. Choose any extension $F\in\mathcal B(L^2(b\mathcal P,\mu))$ of $T$ satisfying $\|F\|=\|T\|$ and $P_{H^2(\mathcal P)}F|_{H^2(\mathcal P)}=T$. One can choose $F=\begin{bmatrix}
    T & 0\\
    0 & 0
\end{bmatrix}$ with respect to  $L^2(b\mathcal P, \mu)=H^2(\mathcal P) \oplus H^2(\mathcal P)^\perp$. For $m\geq 1$, define
\[
S_m=\frac1m\sum_{k=1}^m\Phi^k(F).
\]
Since $\Phi$ is contractive, $\|S_m\| \leq \|F\|
=\|T\|$. Hence, $\{S_m\}$ is bounded and therefore admits a weak operator topology cluster point. Let $M$ be such a cluster point. We claim that each $\Phi^k(F)$ is a compression of $T$, that is, 
\[
P_{H^2(\mathcal P)}\Phi^k(F)|_{H^2(\mathcal P)} = T.
\]
Suppose $R$ is a compression of $T$. Then for $f,g\in H^2(\mathcal P)$,
\begin{align*}
\langle P_{H^2(\mathcal P)}\Phi(R)f,g\rangle
=
\langle \Phi(R)f,g\rangle
&=
\sum_{j=1}^3
\langle RN_jf,N_jg\rangle\\
&=
\sum_{j=1}^3
\langle P_{H^2(\mathcal P)}RN_jf,N_jg\rangle \quad [\text{as each $N_jf, N_jg \in H^2(\mathcal P)$}]\\
&=
\sum_{j=1}^3
\langle TN_jf, N_jg\rangle\\
&=
\frac{1}{2}\langle TM_af,M_ag\rangle + \frac{1}{8}\langle TM_sf, M_sg\rangle+\frac{1}{2}\langle TM_pf, M_pg\rangle\\
&=
\frac{1}{2}\langle TT_af,T_ag\rangle + \frac{1}{8}\langle TT_sf, T_sg\rangle+\frac{1}{2}\langle TT_pf,T_pg\rangle \\
&=
\langle Tf,g\rangle,
\end{align*}
where the last equality follows from the Brown-Halmos identities. Hence, $\Phi(R)$ is again a compression of $T$. Since $F$ is a compression of $T$, it follows inductively that every $\Phi^k(F)$ is a compression of $T$. Consequently, $P_{H^2(\mathcal P)}S_m|_{H^2(\mathcal P)}=T$ for every $m \geq 1$. Passing to the weak operator topology limit, we have that 
\[
P_{H^2(\mathcal P)}M|_{H^2(\mathcal P)}=T.
\]
Next observe that
\[
\Phi(S_m)-S_m
=
\frac1m\bigl(\Phi^{m+1}(F)-\Phi(F)\bigr).
\]
Since $\Phi$ is contractive, $\|\Phi^{m+1}(F)-\Phi(F)\|
\leq
2\|F\|$ and therefore
\[
\|\Phi(S_m)-S_m\|
\leq
\frac{2\|F\|}{m}=\frac{2\|T\|}{m}
\longrightarrow0.
\]
Since norm convergence implies weak operator convergence and $\Phi$ is weak operator continuous, we have that $\Phi(M)=M$. Finally, we apply Proposition 2.4 in \cite{DavieJewell1977} to the commuting normal operators
$N_1, N_2, N_3$. Since $\Phi(M)=M$, that is, $N_1^*MN_1+N_2^*MN_2+N_3^*MN_3=M$, it follows that $MN_j=N_jM$ for $1 \leq j \leq 3$. Consequently, $M$ commutes with $M_a, M_s$ and $M_p$. Since $M$ is a cluster point of $\{S_m\}$ in the weak operator topology, $\lVert M\rVert \le \liminf \lVert S_m\rVert \le \lVert T\rVert$. Together with the reverse inequality $\lVert T\rVert \le \lVert M\rVert$, we conclude $\lVert M\rVert = \lVert T\rVert$. 
\end{proof}

Next, we present a complete description of all Toeplitz operators on $H^2(\mathcal P)$.

\begin{theorem}\label{thm_BH}
    A bounded linear operator $T$ acting on $H^2(\mathcal P)$ is Toeplitz if and only if $T$ satisfies the Brown-Halmos identities. 
\end{theorem}
\begin{proof}
We have already shown that every Toeplitz operator on $H^2(\mathcal P)$
satisfies the Brown-Halmos identities. Conversely, let $T\in\mathcal B(H^2(\mathcal P))$ satisfy the Brown-Halmos identities. By Lemma \ref{lem:lifting}, there exists a
bounded operator $M$ on $L^2(b\mathcal P,\mu)$ such that $\|M\|=\|T\|$, $M$ commutes with $M_a, M_s, M_p$ and $T=P_{H^2(\mathcal P)}M|_{H^2(\mathcal P)}$. By Lemma \ref{lem_402}, there exists a function
$\varphi\in L^\infty(b\mathcal P,\mu)$ such that $M=M_\varphi$. Hence,
\[
T
=
P_{H^2(\mathcal P)}M|_{H^2(\mathcal P)}
=
P_{H^2(\mathcal P)}M_\varphi|_{H^2(\mathcal P)}
=
T_\varphi.
\]
The proof is now complete.
\end{proof}

As applications of the preceding theorem, we obtain the following results.

\begin{corollary}\label{cor_406}
    Let $T \in \mathcal{B}(H^2(\mathcal P))$ commute with $T_a, T_s$ and $T_p$. Then $T$ is Toeplitz. 
\end{corollary}

\begin{proof}
   Since $(T_a, T_s, T_p)$ is a $\mathcal P$-isometry, it follows from \cite{PalTomar2026} that 
   \[
   T_a^*T_a+\frac{1}{4}T_s^*T_s=I, \quad  T_s=T_s^*T_p \quad \text{and} \quad T_p^*T_p=I.
   \]
  As $T$ commutes with $T_a, T_s, T_p$, we have that $T$ acting on $H^2(\mathcal P)$ satisfies the Brown-Halmos identities. The desired conclusion now follows from Theorem \ref{thm_BH}.
\end{proof}

\begin{corollary}
    The only compact Toeplitz operator on $H^2(\mathcal{P})$ is the zero operator.
\end{corollary}

\begin{proof}
   Let $T_{\varphi}$ be a compact Toeplitz operator. By Theorem~\ref{thm_BH}, $(T_p^*)^n T_{\varphi} T_p^n = T_{\varphi}$ for every $n \geq 1$. Since $T_p$ is a pure isometry,
\[
\langle T_p^n f, g \rangle = \langle f, (T_p^*)^n g \rangle \longrightarrow 0 \quad \text{for all } f, g \in H^2(\mathcal{P}),
\]
i.e., $T_p^n f \to 0$ weakly for every $f \in H^2(\mathcal{P})$. It follows from the compactness of $T_{\varphi}$ that $\underset{n}{\lim}\| T_{\varphi} T_p^n f \| \longrightarrow 0$ for all $f \in H^2(\mathcal{P})$. For arbitrary $f, g \in H^2(\mathcal{P})$, we have
\begin{align*}
\langle T_{\varphi} f, g \rangle
= \langle (T_p^*)^n T_{\varphi} T_p^n f, g \rangle 
= \langle T_{\varphi} T_p^n f, T_p^n g \rangle 
\le \| T_{\varphi} T_p^n f \| \cdot \| T_p^n g \| 
= \| T_{\varphi} T_p^n f \| \cdot \| g \|,
\end{align*}
which converges to $0$ as $n \to \infty$. Therefore, it follows that $T_{\varphi} = 0$.
\end{proof}

Next, we determine the joint reducing subspaces of $T_a, T_s, T_p$. Recall that a closed subspace $\mathcal M$ of a Hilbert space
$\mathcal H$ is said to be a \textit{joint reducing subspace} for a
family of operators $\{T_\alpha\}_{\alpha\in I}\subseteq
\mathcal B(\mathcal H)$ if $\mathcal M$ reduces each $T_\alpha$, that is, $T_\alpha\mathcal M\subseteq\mathcal M$ and $T_\alpha^*\mathcal M\subseteq\mathcal M$ for every $\alpha\in I$.

\begin{lemma} \label{lem:irreducibility}
 $T_a, T_s, T_p$ acting on $H^2(\mathcal P)$
has no nontrivial joint reducing subspaces. 
\end{lemma}

\begin{proof}
Let $\mathcal M$ be a joint reducing subspace for
$(T_a,T_s,T_p)$ and let $P_{\mathcal M}$ denote the orthogonal
projection of $H^2(\mathcal P)$ onto $\mathcal M$. Since $\mathcal M$
reduces each of $T_a,T_s$ and $T_p$, we have that
$P_{\mathcal M}T_\nu=T_\nu P_{\mathcal M}$ and
$P_{\mathcal M}T_\nu^*=T_\nu^*P_{\mathcal M}$ whenever $\nu\in\{a,s,p\}$. Recall that the constant function $1$ is the degree-zero homogeneous component of $H^2(\mathcal P)$. Since multiplication by each of $a,s,p$ raises the $(1,1,2)$-homogeneous degree, the orthogonality of the homogeneous decomposition gives
\[
T_a^*1=T_s^*1=T_p^*1=0.
\]
Let $g \in \ker T_a^*\cap\ker T_s^*\cap\ker T_p^*$. Then $\langle g,a^kf\rangle
=\langle (T_a^*)^kg,f\rangle=0$ and similarly
$
\langle g,sf\rangle=\langle g,pf\rangle=0$ for
every $f\in H^2(\mathcal P)$ and $k \geq 1$. Consequently, 
$
\left\langle
g,a^\alpha s^\beta p^\gamma f
\right\rangle=0
$
for every $\alpha,\beta,\gamma\geq0$ with
$\alpha+\beta+\gamma>0$ and every $f\in H^2(\mathcal P)$.
In particular, taking $f=1$, we see that $g$ is orthogonal to every
nonconstant monomial. Write the weighted homogeneous decomposition of $g \in H^2(\mathcal P)$ as
\[
g=\sum_{n=0}^{\infty}g_n,
\qquad g_n\in\mathcal Q_n,
\]
where $\mathcal Q_n=\mathrm{span}\{a^\alpha s^\beta p^\gamma: \alpha, \beta, \gamma \geq 0, \alpha+\beta+2\gamma=n\}$ for $n \geq 1$. Then $g\perp\mathcal Q_n$ for every $n\geq1$.
By the orthogonality of the weighted homogeneous decomposition,
$0=\langle g,g_n\rangle=\|g_n\|^2$ and thus $g_n=0$ for every $n\geq1$. Hence, $g=g_0\in\mathcal Q_0=\mathbb C1$.
Since $1\in\ker T_a^*\cap\ker T_s^*\cap\ker T_p^*$, we have that
\[
\ker T_a^*\cap\ker T_s^*\cap\ker T_p^*
=\mathbb C1.
\]
Set $h=P_{\mathcal M}1$. Using the commutativity of $P_{\mathcal M}$ with the adjoints and the equalities $T_a^*1=T_s^*1=T_p^*1=0$, it follows that $T_\nu^*h
=T_\nu^*P_{\mathcal M}1
=P_{\mathcal M}T_\nu^*1
=0$ whenever $\nu\in\{a,s,p\}$. Hence, $h$ belongs to the joint kernel and so, $P_{\mathcal M}1=c1$ for some $c\in\mathbb C$. Since $P_{\mathcal M}$ commutes with $T_a,T_s$ and $T_p$, it commutes
with every polynomial in these multiplication operators. Therefore,
for every $q\in\mathbb C[a,s,p]$,
\[
P_{\mathcal M}q
=P_{\mathcal M}q(T_a,T_s,T_p)1
=q(T_a,T_s,T_p)P_{\mathcal M}1
=cq.
\]
Since the polynomial algebra is dense in $H^2(\mathcal P)$ and
$P_{\mathcal M}$ is bounded, we have that $P_{\mathcal M}=cI$ on $H^2(\mathcal P)$. As $P_{\mathcal M}$ is an orthogonal projection, it follows that either $P_{\mathcal M}=0$ or $P_{\mathcal M}=I$. Therefore, $\mathcal M=\{0\}$ or $\mathcal M=H^2(\mathcal P)$.
\end{proof}

The proof of Lemma \ref{lem:irreducibility} can be applied to any operator that doubly commutes with the coordinate multipliers $T_a, T_s, T_p$. The only modification is to replace the
orthogonal projection $P_{\mathcal M}$ by an arbitrary operator
$X\in\mathcal B(H^2(\mathcal P))$ and to use $XT_\nu^*=T_\nu^*X$ for  $\nu\in\{a,s,p\}$. In particular, if $h=X1$, then
$
T_\nu^*h
=T_\nu^*X1
=XT_\nu^*1
=0
$
for $\nu\in\{a,s,p\}$. The proof of Lemma \ref{lem:irreducibility} then gives $X1=c1$ for some $c\in\mathbb C$. By continuity arguments, we arrive at the following result.

\begin{theorem}
If $X\in\mathcal B(H^2(\mathcal P))$ doubly commutes with the coordinate multipliers $T_a,T_s$ and $T_p$ acting on $H^2(\mathcal P)$, then $X=cI$ for some $c\in\mathbb C$.
\end{theorem}

Classically, Toeplitz operators on the Hardy space over the unit disc \cite{BrownHalmos1964} are characterized in terms of the coordinate multiplication operator, which is the canonical isometry associated with the disc. This viewpoint naturally leads to the study of Toeplitz operators on Hardy spaces over more general domains by replacing the coordinate multiplication operator with the canonical isometric tuple associated with the underlying domain. Such generalizations have been carried out, for instance, for the Euclidean unit ball \cite{DavieJewell1977} and the symmetrized polydisc \cite{DasSau2021}. Motivated by these, we put forth the following definition. 

\begin{definition}
Let $(A, S, P)$ be a $\mathcal P$-isometry acting on a Hilbert space $\mathcal H$. We say that $T \in \mathcal B(\mathcal H)$ satisfies the Brown-Halmos relations with respect to $(A, S, P)$ if 
\[
P^*TP=T, \quad S^*TP=TS \quad \text{and} \quad A^*TA+\frac{1}{4}S^*TS=T.
\]
\end{definition}

Theorem \ref{thm_BH} tells us that a bounded operator $T$ on $H^2(\mathcal P)$ is Toeplitz if and only if $T$ satisfies the Brown-Halmos relations with respect to the $\mathcal{P}$-isometry $(T_a, T_s, T_p)$. A key ingredient in its proof is Lemma \ref{lem:lifting}, which is a commutant lifting type result. We conclude this section with the following generalization of this lifting result.

\begin{theorem}[Commutant lifting theorem]
Let $(T_1, T_2, T_3)$ be a $\mathcal P$-isometry on a Hilbert space
$\mathcal H$ and let $(N_1, N_2, N_3)$ be its minimal $\mathcal P$-unitary extension acting on a Hilbert space $\mathcal K \supseteq \mathcal H$. Then $T \in \mathcal B(\mathcal H)$ satisfies the Brown-Halmos
relations with respect to $(T_1, T_2, T_3)$ if and only if  there exists a unique operator $N \in \mathcal B(\mathcal K)$ such that
\[
T=P_{\mathcal H}N|_{\mathcal H}, \quad \|N\|=\|T\| \quad \text{and} \quad NN_j=N_jN \ \ (1 \leq j \leq 3).
\]
\end{theorem}

\begin{proof}
The argument for the forward implication is similar to the proof of Lemma  \ref{lem:lifting}. Indeed, one only replaces the canonical pentablock isometry $(T_a,T_s,T_p)$
on $H^2(\mathcal P)$ by the arbitrary pentablock isometry $(T_1,T_2,T_3)$, the triple $(M_a,M_s,M_p)$ on $L^2(b\mathcal P,\mu)$ by the minimal $\mathcal P$-unitary extension $(N_1,N_2,N_3)$ and the ambient space $L^2(b\mathcal P,\mu)$ by the extension space $\mathcal K$. Since the proof of Lemma \ref{lem:lifting} uses only the Brown-Halmos identities, the existence of the minimal $\mathcal P$-unitary extension, the associated commuting normal spherical isometry
\[
\left(\frac{1}{\sqrt2}N_1,\,
\frac{1}{2\sqrt2}N_2,\,
\frac{1}{\sqrt2}N_3\right),
\]
and Proposition 2.4 of \cite{DavieJewell1977}, the same argument carries over verbatim to the present setting. This yields an operator $N\in \mathcal B(\mathcal K)$ with the desired properties. To prove the uniqueness, assume that $Z_1$ and
$Z_2$ are two operators satisfying the conclusions of the
theorem. Set $Z=Z_1-Z_2$. Then $Z$ commutes with $N_1, N_2, N_3$ and $P_{\mathcal H}Z|_{\mathcal H}=0$. We first show that $Zh=0$ for every $h\in\mathcal H$. Indeed, let
$h,k\in\mathcal H$ and let
$\alpha=(\alpha_1,\alpha_2,\alpha_3)\in\mathbb Z_+^3$. Since $Z$
commutes with each $N_i$, we have
\[
\begin{aligned}
\langle Zh,
N_1^{*\alpha_1}N_2^{*\alpha_2}N_3^{*\alpha_3}k\rangle
=
\langle
N_1^{\alpha_1}N_2^{\alpha_2}N_3^{\alpha_3}Zh,
k
\rangle
&=
\langle
ZN_1^{\alpha_1}N_2^{\alpha_2}N_3^{\alpha_3}h,
k
\rangle\\
&=
\langle
P_{\mathcal H}ZN_1^{\alpha_1}N_2^{\alpha_2}N_3^{\alpha_3}h,
k
\rangle\\
&=0,
\end{aligned}
\]
where the last equality holds as $N_1^{\alpha_1}N_2^{\alpha_2}N_3^{\alpha_3}h \in \mathcal H$ and $P_{\mathcal H}Z|_{\mathcal H}=0$. By minimality, 
\[
\mathcal K=\overline{\operatorname{span}}
\left\{
N_1^{*\alpha_1}N_2^{*\alpha_2}N_3^{*\alpha_3}k:
k\in\mathcal H,\,
\alpha_1, \alpha_2, \alpha_3 \in\mathbb Z_+
\right\}.
\]
Hence, $Zh=0$ for every $h\in\mathcal H$. Finally, for every $h\in\mathcal H$ and
$\alpha_1, \alpha_2, \alpha_3 \in\mathbb Z_+^3$,
$
ZN_1^{*\alpha_1}N_2^{*\alpha_2}N_3^{*\alpha_3}h
=
N_1^{*\alpha_1}N_2^{*\alpha_2}N_3^{*\alpha_3}Zh
=0
$
and thus $Z=0$.

\smallskip 

Conversely, suppose there exists a unique operator $N \in \mathcal B(\mathcal K)$ such that $N$ commutes with $N_1, N_2, N_3$, $\|N\|=\|T\|$ and $T=P_{\mathcal H}N|_{\mathcal H}$. Since $(N_1, N_2, N_3)$ is a $\mathcal P$-unitary, we have by Theorem 4.2 in \cite{PalTomar2026} that
\[
N_1^*N_1+\frac{1}{4}N_2^*N_2=I, \quad N_2^*N_3=N_2 \quad \text{and} \quad N_3^*N_3=I.
\]
Since $N$ commutes with $N_1, N_2, N_3$, it follows that $N$ satisfies the Brown-Halmos relations with respect to $(N_1, N_2, N_3)$, that is,
\begin{align}\label{eqn_602}
N_1^*NN_1+\frac{1}{4}N_2^*NN_2=N, \quad N_2^*NN_3=NN_2 \quad \text{and} \quad N_3^*NN_3=N.
\end{align}
Furthermore, with respect to the decomposition $\mathcal{K}=\mathcal{H}\oplus \mathcal{H}^\perp$, we can write
\[
N_1=\begin{bmatrix}
    T_1 & *\\
    0 & *
\end{bmatrix}, \quad N_2=\begin{bmatrix}
    T_2 & *\\
    0 & *
\end{bmatrix}
, \quad N_3=\begin{bmatrix}
    T_3 & *\\
    0 & *
\end{bmatrix}
\quad \text{and} \quad N=\begin{bmatrix}
    T & *\\
    * & *
\end{bmatrix}.
\]
It now directly follows from \eqref{eqn_602} that $T$ satisfies the Brown-Halmos relations with respect to $(T_1, T_2, T_3)$, which completes the proof. 
\end{proof}

\section{Weighted Hardy spaces on the pentablock}\label{sec_05}

\noindent In this section, we consider two different weighted Hardy type spaces on $\mathcal P$, where the weights comes from the Jacobian of the parametrization $\Psi$. Recall the holomorphic map $\Psi:\mathbb D^3\longrightarrow \mathcal P$ given by
\[
\Psi(\xi,\lambda_1,\lambda_2)
=
\left(
\frac{\xi(\lambda_1-\lambda_2)}{2},
\lambda_1+\lambda_2,
\lambda_1\lambda_2
\right).
\]
Its extension to $\mathbb T^3$ parametrizes the distinguished boundary
$b\mathcal P$. Let $
\pi_{\text{symm}}(\lambda_1,\lambda_2)
=
(\lambda_1+\lambda_2,\lambda_1\lambda_2)$
be the symmetrization map. The complex derivative of $\Psi$ is
\[
D\Psi(\xi,\lambda_1,\lambda_2)
=
\begin{pmatrix}
\dfrac{\lambda_1-\lambda_2}{2}
&
\dfrac{\xi}{2}
&
-\dfrac{\xi}{2}
\\[6pt]
0&1&1\\
0&\lambda_2&\lambda_1
\end{pmatrix}.
\]
Consequently, the complex Jacobian of $\Psi$ is
\begin{align*}
J_\Psi(\xi,\lambda_1,\lambda_2)
=
\det D\Psi(\xi,\lambda_1,\lambda_2)
=
\frac{\lambda_1-\lambda_2}{2}
\det
\begin{pmatrix}
1&1\\
\lambda_2&\lambda_1
\end{pmatrix}
=
\frac{(\lambda_1-\lambda_2)^2}{2}.
\end{align*}
Thus, $J_\Psi$ is independent of $\xi$. The Jacobian $J_\Psi$ gives rise to two natural weighted boundary
measures, corresponding to the weights $|J_\Psi|$ and $|J_\Psi|^2$.
We consider precisely these two weights for geometric reasons. First,
the weight $|J_\Psi|$ is directly related to the symmetrization map
$\pi_{\mathrm{symm}}:\mathbb D^2\to\mathbb G_2$. Indeed,
\[
    J_{\pi_{\mathrm{symm}}}(\lambda_1,\lambda_2)
    =\lambda_1-\lambda_2,
    \qquad
    |J_\Psi|
    =\frac12|J_{\pi_{\mathrm{symm}}}|^2.
\]
Thus, the $|J_\Psi|$-weighted construction naturally recovers the
Jacobian-weighted Hardy space of the symmetrized bidisc defined as in \cite{MisraShyamRoyZhang2013} on the slice
$a=0$. This makes $|J_\Psi|$ the natural weight for relating the
pentablock construction to $\mathbb G_2$. The second choice $|J_\Psi|^2$ is the usual Jacobian-squared density
that arises in the change-of-variables formula for proper holomorphic
maps between domains of the same complex dimension. Although $\Psi$
does not satisfy the proper finite-to-one hypotheses in the present
setting, this weight is nevertheless a natural geometric quantity
associated with its complex Jacobian. We therefore restrict our
attention to these two weights. Of course, higher powers of
$|J_\Psi|$ could be introduced formally, but they do not have a geometric interpretation relevant to the present
construction.

\subsection{A Jacobian-weighted Hardy space on the pentablock} For the normalized Haar measure $m_k$ on $\mathbb T^k$, we have
\begin{align*}
\int_{\mathbb T^3}|J_\Psi|\,dm_3
=
\frac{1}{2}
\int_{\mathbb T^2}
|\lambda_1-\lambda_2|^2
\,dm_2(\lambda_1,\lambda_2)
=
1
\end{align*}
since $|\lambda_1-\lambda_2|^2
=
2-\lambda_1\overline{\lambda}_2
-\overline{\lambda}_1\lambda_2$ and the nonconstant Fourier terms have integral zero. One can consider the probability measure on $b\mathcal P$ obtained
by pushing forward the measure $|J_\Psi|\,dm_3$. More precisely, define
$\mu_J$ on $b\mathcal P$ by
\[
\int_{b\mathcal P}F\,d\mu_J
=
\int_{\mathbb T^3}
F\circ\Psi\,|J_\Psi|\,dm_3=\frac12
\int_{\mathbb T^3}
F\circ\Psi\,
|\lambda_1-\lambda_2|^2\,dm_3
\]
for every continuous function $F$ on $b\mathcal P$. Equivalently,
$\mu_J
=
\Psi_*\bigl(|J_\Psi|\,m_3\bigr)$. This gives rise to a natural weighted Hardy space on $\mathcal P$. We define
\[
H_J^2(\mathcal P)
=
\overline{\mathbb C[a,s,p]}^{\,L^2(b\mathcal P,\mu_J)}.
\]
The same arguments as in Section \ref{sec_02} gives a radial realization of the Jacobian-weighted
space $H_J^2(\mathcal P)$. The only modification is the presence of the weight $|J_\Psi|$. Evidently, 
\[
J_\Psi(r\xi,r\lambda_1,r\lambda_2)
=
\frac{r^2|\lambda_1-\lambda_2|^2}{2}
=
r^2J_\Psi(\xi,\lambda_1,\lambda_2).
\]
We denote by $\Psi_r=\delta_r\circ \Psi$ for $0<r<1$. Let $H_{J,\mathrm{rad}}^2(\mathcal P)$ be the space of holomorphic
functions $f$ on $\mathcal P$ for which
\[
\|f\|_{H_{J,\mathrm{rad}}^2(\mathcal P)}
=
\sup_{0<r<1}
\left(
\int_{\mathbb T^3}
|f\circ\Psi_r|^2
|J_\Psi(r\xi,r\lambda_1,r\lambda_2)|
\,dm_3
\right)^{1/2}
<\infty.
\]
Following the proof of Theorem \ref{thm:radial-realization}, with the
same weighted homogeneous decomposition and using the orthogonality of
the spaces $\mathcal Q_n$, one can prove that 
\[
f=\sum_{n\geq0}q_n \quad \text{and} \quad \int_{\mathbb T^3}
|f\circ\Psi_r|^2
|J_\Psi(r\xi,r\lambda_1,r\lambda_2)|
\,dm_3
=
r^2\sum_{n=0}^{\infty}
r^{2n}
\int_{\mathbb T^3}
|q_n\circ\Psi|^2
|J_\Psi|\,dm_3.
\]
Here the additional factor $r^2$ comes precisely from the homogeneity of the Jacobian $J_\Psi$. Since the terms in the series are nonnegative,
\[
\sup_{0<r<1}
r^2\sum_{n=0}^{\infty}
r^{2n}
\int_{\mathbb T^3}
|q_n\circ\Psi|^2|J_\Psi|\,dm_3
=
\sum_{n=0}^{\infty}
\int_{\mathbb T^3}
|q_n\circ\Psi|^2|J_\Psi|\,dm_3.
\]
Thus the radial norm
coincides with the boundary norm, that is,
\[
\|f\|^2_{H_{J,\mathrm{rad}}^2(\mathcal P)}
=
\|f\|^2_{H_J^2(\mathcal P)}=\overset{\infty}{\underset{n=0}{\sum}}\|q_n\|^2_{H_J^2(\mathcal P)}.
\]
The converse implication follows in exactly the same manner as in
Theorem~\ref{thm:radial-realization}, by applying Parseval's identity to
the $(1,1,2)$-quasi-circular action. Hence the polynomial closure
definition and the radial definition determine the same Hilbert space $H^2_J(\mathcal P)$ isometrically.

\begin{proposition}
    The spaces $H^2_J(\mathcal P)$ and $H_{J,\mathrm{rad}}^2(\mathcal P)$ are isometrically isomorphic. 
\end{proposition}

The choice of the weight $|J_\Psi|$ is particularly natural from the
perspective of the symmetrized bidisc. For the symmetrization map $\pi_{\text{symm}}$, the associated Jacobian map is $
J_{\pi_{\text{symm}}}(\lambda_1,\lambda_2)
=
\lambda_1-\lambda_2$. Consequently,
\[
|J_\Psi(\xi,\lambda_1,\lambda_2)|
=
\frac{1}{2}
|J_{\pi_{\mathrm{symm}}}(\lambda_1,\lambda_2)|^2,
\]
which suggests a natural connection between the weighted Hardy space associated with $\Psi$ and the Hardy space of the symmetrized bidisc. Recall from \cite{MisraShyamRoyZhang2013} that the Hardy space
$H^2(\mathbb G_2)$ consists of holomorphic functions on $\mathbb G_2$
for which
\[
\|g\|_{H^2(\mathbb G_2)}^2
=
\frac{1}{2}
\sup_{0<r<1}
\int_{\mathbb T^2}
\left|g\circ\pi_{\mathrm{symm},r(\lambda_1, \lambda_2)}\right|^2
\left|J_{\pi_{\mathrm{symm}}}(r\lambda_1,r\lambda_2)\right|^2
\,dm_2,
\]
where $\pi_{\mathrm{symm},r}(\lambda_1,\lambda_2)
=
\left(r(\lambda_1+\lambda_2),r^2\lambda_1\lambda_2\right)$ and $m_2$ is the normalized Haar measure on $\mathbb T^2$. Since
$J_{\pi_{\mathrm{symm}}}(\lambda_1,\lambda_2)
=
\lambda_1-\lambda_2$, we have that $\left|J_{\pi_{\mathrm{symm}}}(r\lambda_1,r\lambda_2)\right|^2
=
r^2|\lambda_1-\lambda_2|^2$. Now let $g\in H^2(\mathbb G_2)$ and define
$f(a,s,p)=g(s,p)$. Then $f\in\operatorname{Hol}(\mathcal P)$ and
$
f\circ\Psi_r=g\circ\pi_{\mathrm{symm},r}$. Consequently,
\begin{align*}
\|f\|_{H_J^2(\mathcal P)}^2
=
\frac{1}{2}
\sup_{0<r<1}
r^2
\int_{\mathbb T^2}
\left|g\circ\pi_{\mathrm{symm},r}\right|^2
|\lambda_1-\lambda_2|^2\,dm_2
=
\|g\|_{H^2(\mathbb G_2)}^2.
\end{align*}
Consequently, we have the following result.

\begin{proposition}
The map
\[
V:H^2(\mathbb G_2)\longrightarrow H_J^2(\mathcal P),
\qquad
(Vg)(a,s,p)=g(s,p),
\]
is an isometry.     
\end{proposition}

Geometrically, the above result is analogous to embedding of the symmetrized bidisc into $\mathcal P$  via the the slice
$
\mathcal P\cap(\{0\}\times\mathbb C^2)
=
\{(0,s,p):(s,p)\in\mathbb G_2\}$. 

\smallskip 

Next, we compare the space $H_J^2(\mathcal P)$ with the unweighted Hardy space $H^2(\mathcal P)$. Recall that the Hardy space norms are defined by
\[
\|f\|_{H^2(\mathcal P)}^2
=
\sup_{0<r<1}\int_{\mathbb T^3}|f\circ\Psi_r|^2\,dm_3 \ \ \text{and} \ \
\|f\|_{H_J^2(\mathcal P)}^2
=
\frac12\sup_{0<r<1}r^2
\int_{\mathbb T^3}
|f\circ\Psi_r|^2
|\lambda_1-\lambda_2|^2\,dm_3.
\]
Consequently, $\|f\|_{H_J^2(\mathcal P)}
\leq
\sqrt{2}\,
\|f\|_{H^2(\mathcal P)}$. Thus the identity map on holomorphic functions induces a continuous
inclusion map $\text{id}: H^2(\mathcal P)
\hookrightarrow
H_J^2(\mathcal P)$. In fact, our next result shows that the inclusion is proper.

\begin{proposition}\label{prop:equiv norms}
The inclusion of $H^2(\mathcal P)$ into $H_J^2(\mathcal P)$ is proper. Moreover, the two norms $\|.\|_{H^2(\mathcal P)}$ and $\|.\|_{H^2_J(\mathcal P)}$ are not equivalent.   
\end{proposition}

\begin{proof}  Consider
\[
f(a,s,p)=\frac{1}{2-s}
=\frac12\frac{1}{1-s/2}
=\sum_{n=0}^{\infty}\frac{s^n}{2^{n+1}}.
\]
Since $s^n\in\mathcal Q_n$, this is precisely the weighted homogeneous
decomposition of $f$. We can therefore use the orthogonality of the
spaces $\mathcal Q_n$ to determine its membership in the two Hardy
spaces. Put
\[
q_n(a,s,p)=\frac{s^n}{2^{n+1}},\qquad n\geq0.
\]
Writing $\lambda_1=e^{i(t+\theta)},
\lambda_2=e^{i(t-\theta)}$, we have that $
|\lambda_1+\lambda_2|=2|\cos\theta|$. Consequently,
\[
\|q_n\|_{H^2(\mathcal P)}^2
=
\frac{1}{2^{2n+2}}
\int_{\mathbb T^2}
|\lambda_1+\lambda_2|^{2n}\,dm_2(\lambda_1,\lambda_2)
=
\frac12\int_{-\pi}^{\pi}|\cos\theta|^{2n}\,d\theta.
\]
For $|\theta|\leq 1/\sqrt{n+1}$, note that
$\displaystyle
|\cos\theta|=\cos\theta\geq 1-\frac{\theta^2}{2}
\geq 1-\frac{1}{2(n+1)}$. Therefore
\[
\int_{-\pi}^{\pi}|\cos\theta|^{2n}\,d\theta
\geq
\frac{2}{\sqrt{n+1}}
\left(1-\frac{1}{2(n+1)}\right)^{2n}.
\]
Since
\[
\left(1-\frac{1}{2(n+1)}\right)^{2n}\longrightarrow e^{-1},
\]
there exists $N$ such that the last factor is at least $1/(2e)$ for $n\geq N$. Thus 
\[
\sum_{n \geq N}^{\infty}
\|q_n\|_{H^2(\mathcal P)}^2
\geq
\frac{1}{2e}\sum_{n \geq N}^{\infty}\frac{1}{\sqrt{n+1}}
=\infty
\]
and so, $f=\sum_{n=0}^{\infty}q_n \notin H^2(\mathcal P)$. On the other hand, the Jacobian-weighted Hardy space norm is given by
\begin{align*}
\|q_n\|_{H_J^2(\mathcal P)}^2
=
\frac{1}{2^{2n+2}}\frac12
\int_{\mathbb T^2}
|\lambda_1+\lambda_2|^{2n}
|\lambda_1-\lambda_2|^2
\,dm_2(\lambda_1,\lambda_2)
&=
\frac12\int_{-\pi}^{\pi}
|\cos\theta|^{2n}\sin^2\theta\,d\theta\\
&=
2\int_0^{\pi/2}
\cos^{2n}\theta\,\sin^2\theta\,d\theta.
\end{align*}
For $0\leq\theta<\pi/2$, the inequality
$\cos\theta\leq e^{-\theta^2/2}$ follows from $\log(\cos\theta)\leq-\theta^2/2$. Together with
$\sin\theta\leq\theta$, it follows that
\[
\|q_n\|_{H_J^2(\mathcal P)}^2
\leq
2\int_0^\infty \theta^2e^{-n\theta^2}\,d\theta
=
\frac{2}{n^{3/2}}\int_0^\infty u^2e^{-u^2}\,du
=
\frac{\sqrt{\pi}}{2n^{3/2}}.
\]
Consequently,
\[
\sum_{n=1}^{\infty}
\|q_n\|_{H_J^2(\mathcal P)}^2
\leq
\frac{\sqrt{\pi}}{2}
\sum_{n=1}^{\infty}\frac{1}{n^{3/2}}
<\infty
\]
and thus, $f \in H_J^2(\mathcal P)$. Therefore, $H^2(\mathcal P)\subsetneq H_J^2(\mathcal P)$ and so, $H^2(\mathcal P)$ as a vector space admits two different norms. We show that the two norms $\|.\|_{H^2(\mathcal P)}$ and $\|.\|_{H^2_J(\mathcal P)}$ are not equivalent. Consider the
functions $f_n(a,s,p)=s^n$ for $n \geq 1$. Then
\[
\|f_n\|_{H^2(\mathcal P)}^2
=
\int_{\mathbb T^2}
|\lambda_1+\lambda_2|^{2n}\,dm_2 \quad \text{and} \quad 
\|f_n\|_{H_J^2(\mathcal P)}^2
=
\frac12
\int_{\mathbb T^2}
|\lambda_1+\lambda_2|^{2n}
|\lambda_1-\lambda_2|^2\,dm_2.
\]
Writing $\lambda_1=e^{i(\theta+\phi)}$ and $\lambda_2=e^{i(\theta-\phi)}$, we have that
$\lambda_1+\lambda_2=2e^{i\theta}\cos\phi$ and
$|\lambda_1-\lambda_2|^2=4\sin^2\phi$. Then
\begin{align*}
\frac{\|f_n\|_{H_J^2(\mathcal P)}^2}
{\|f_n\|_{H^2(\mathcal P)}^2}
&=
2\left(\displaystyle\int_0^{2\pi}|\cos\phi|^{2n}\sin^2\phi\,d\phi\right)
\left({\displaystyle\int_0^{2\pi}|\cos\phi|^{2n}\,d\phi}\right)^{-1}
=
\frac{1}{n+1},
\end{align*}
where the last equality follows from the standard beta-integral identity. Thus,
\[
\frac{\|f_n\|_{H_J^2(\mathcal P)}}
{\|f_n\|_{H^2(\mathcal P)}}
=
\frac{1}{\sqrt{n+1}}
\longrightarrow 0.
\]
Thus, the two norms $\|.\|_{H^2(\mathcal P)}$ and $\|.\|_{H^2_J(\mathcal P)}$ are not equivalent.
\end{proof}

\subsection{A Jacobian-squared weighted Hardy space on the pentablock} The preceding construction suggests another Hardy-type space on the
pentablock, obtained by using the square of the modulus of the complex
Jacobian of the map $\Psi$. This is a more customary
Jacobian weight arising in change-of-variables formulas. By above discussion, it follows that
\[
|J_\Psi(\xi,\lambda_1,\lambda_2)|^2
=
\frac14|\lambda_1-\lambda_2|^4 \quad \text{and} \quad 
\int_{\mathbb T^3}|J_\Psi|^2\,dm_3
=
\frac14
\int_{\mathbb T^2}
|\lambda_1-\lambda_2|^4
\,dm_2(\lambda_1,\lambda_2)=\frac{3}{2}.
\]
Define a probability measure $\mu_{J^2}$ on $b\mathcal P$ by
\[
\int_{b\mathcal P}F\,d\mu_{J^2}
=
\frac23
\int_{\mathbb T^3}
F\circ\Psi\,|J_\Psi|^2\,dm_3=\frac16
\int_{\mathbb T^3}
F\circ\Psi\,
|\lambda_1-\lambda_2|^4\,dm_3
\]
for every continuous function $F$ on $b\mathcal P$. Equivalently,
$
\mu_{J^2}
=
(2\slash 3)
\Psi_*\bigl(|J_\Psi|^2m_3\bigr)$. The measure $\mu_{J^2}$ gives rise to the Hardy space 
\[
H_{J^2}^2(\mathcal P)
=
\overline{\mathbb C[a,s,p]}^{\,L^2(b\mathcal P,\mu_{J^2})}.
\]
As before, this definition can be expressed in terms of radial
approximations. Accordingly, one can show that $\|f\|_{H_{J^2}^2(\mathcal P)}^2$ equals
\[
\frac23
\sup_{0<r<1}
\int_{\mathbb T^3}
|f\circ\Psi_r|^2
|J_\Psi(r\xi,r\lambda_1,r\lambda_2)|^2\,dm_3
=
\sup_{0<r<1}
\frac{r^4}{6}
\int_{\mathbb T^3}
|f\circ\Psi_r|^2
|\lambda_1-\lambda_2|^4\,dm_3.
\]
Next, we compare $H_{J^2}^2(\mathcal P)$ with the space $H_J^2(\mathcal P)$. Recall that 
\[
\|f\|_{H_J^2(\mathcal P)}^2
=
\frac12
\sup_{0<r<1}r^2\int_{\mathbb T^3}
|f\circ\Psi_r|^2
|\lambda_1-\lambda_2|^2\,dm_3.
\]
It is not difficult to see that  
\[
\|f\|_{H_{J^2}^2(\mathcal P)}
\leq
\frac{2}{\sqrt{3}}\,
\|f\|_{H_J^2(\mathcal P)}
\]
for every $f \in H^2_J(\mathcal P)$. Thus there is a continuous inclusion of $H_J^2(\mathcal P)$ into the space $
H_{J^2}^2(\mathcal P)$. Consequently, we have the continuous inclusions given by
\[
H^2(\mathcal P)
\hookrightarrow
H^2_J(\mathcal P)
\hookrightarrow
H_{J^2}^2(\mathcal P)
\]
and  $\sqrt{3}\|f\|_{H_{J^2}^2(\mathcal P)} \leq 2\|f\|_{H^2_{J}(\mathcal P)} \leq 2 \sqrt{2} \|f\|_{H^2(\mathcal P)}$ for every $f \in H^2(\mathcal P)$. The last inequality together with Proposition \ref{prop:equiv norms} implies that the two norms $\|.\|_{H^2(\mathcal P)}$ and $\|.\|_{H^2_{J^2}(\mathcal P)}$ are not equivalent on $H^2(\mathcal P)$ as a vector space.  We next show that the inclusion
$H_J^2(\mathcal P)\subset H_{J^2}^2(\mathcal P)$ is proper. For $|s|<2$, consider the map $f(a, s, p)=1\slash (2-s)^{3\slash 2}$. In this case, $\operatorname{Re}(2-s)>0$ and the principal branch of
$(2-s)^{-3/2}$ is holomorphic. Hence, $f$ is holomorphic on $\mathcal P$ and we can write
\[
f(a,s,p)=
\frac{1}{2^{3/2}}
\left(1-\frac{s}{2}\right)^{-3/2}
=
\sum_{n=0}^{\infty}q_n(a,s,p),
\quad 
\text{where}
\quad 
q_n(a,s,p)
=
\frac{(3/2)_n}{2^{n+3/2}n!}\,s^n.
\]
Here $(\alpha)_n$ denotes the Pochhammer symbol. Then
$(3/2)_n=\Gamma(n+3/2) \slash \Gamma(3/2)$, where $\Gamma$ denotes the Gamma function. We first compute the $H_J^2(\mathcal P)$-norm of $s^n$. By definition,
\[
\|s^n\|_{H_J^2(\mathcal P)}^2
=
\frac12
\int_{\mathbb T^2}
|\lambda_1+\lambda_2|^{2n}
|\lambda_1-\lambda_2|^2
\,dm_2(\lambda_1,\lambda_2).
\]
Write
$\lambda_1=e^{it}$ and $\lambda_2=e^{i(t-\theta)}$. Since normalized
Haar measure is invariant under this change of variables, we obtain
\[
\begin{aligned}
\|s^n\|_{H_J^2(\mathcal P)}^2
&=
\frac12
\int_{\mathbb T}
4^n\cos^{2n}\frac{\theta}{2}
\,4\sin^2\frac{\theta}{2}
\,dm_1(\theta)
=
2^{2n+1}
\int_{\mathbb T}
\cos^{2n}\frac{\theta}{2}
\sin^2\frac{\theta}{2}
\,dm_1(\theta).
\end{aligned}
\]
Using the standard beta integral for $a, b>0$, we have
\[
\int_0^{\pi/2}
\sin^{2a-1}t\cos^{2b-1}t\,dt
=
\frac{\Gamma(a)\Gamma(b)}
{2\Gamma(a+b)}
 \quad  \text{and so,} \quad
\|s^n\|_{H_J^2(\mathcal P)}^2
=
\frac{4^n\Gamma(n+1/2)}
{\sqrt{\pi}\Gamma(n+2)}.
\]
We now recall the standard asymptotic formula for ratios of Gamma
functions. For
$\alpha,\beta \geq 0$, there exist constants $c_{\alpha,\beta}, C_{\alpha, \beta}>0$
 and $N\in\mathbb{N}$ such that for every $n \geq N$,
\[
c_{\alpha,\beta}n^{\alpha-\beta} \leq \frac{\Gamma(n+\alpha)}{\Gamma(n+\beta)}
\leq
C_{\alpha,\beta}n^{\alpha-\beta}.
\]
Consequently, for sufficiently large $n$, there exists $c_1, c_2>0$ such that 
\begin{align*}
\|q_n\|_{H_J^2(\mathcal P)}^2
&=
\frac{\Gamma(n+3/2)^2}
{2^{2n+3}\Gamma(3/2)^2\Gamma(n+1)^2}
\frac{4^n\Gamma(n+1/2)}
{\sqrt{\pi}\Gamma(n+2)}\\
&=
\frac{\Gamma(n+3/2)^2\Gamma(n+1/2)}
{8\Gamma(3/2)^2\sqrt{\pi}\,
\Gamma(n+1)^2\Gamma(n+2)}\\
&=\frac{(n+1/2)^2\Gamma(n+1/2)^3}
{8\Gamma(3/2)^2\sqrt{\pi}\,
\Gamma(n+1)^2\Gamma(n+2)}\\
& \geq \frac{n^2}{8\sqrt{\pi}\Gamma(3/2)^2}\left(\frac{\Gamma\left(n+1/2\right)}{\Gamma(n+1)}\right)^2
\left(\frac{\Gamma\left(n+1/2\right)}
{\Gamma(n+2)}\right) \quad [\text{as} (n+1/2)^2 \geq n^2]\\
& \geq \frac{n^2}{8\sqrt{\pi}\Gamma(3/2)^2}(c_1n^{-1/2})^2c_2n^{-3/2} \\ 
& \qquad \qquad \qquad \quad \quad [\text{by asymptotic formula for ratios of Gamma functions}]\\
& \geq \frac{c_1^2c_2}
{8\Gamma(3/2)^2\sqrt{\pi}}\,
n^{-1/2} 
\end{align*}
and so,  $\sum_{n=0}^{\infty}
\|q_n\|_{H_J^2(\mathcal P)}^2
=
\infty$. By the orthogonal decomposition and the radial realization of
$H_J^2(\mathcal P)$, it follows that
$f\notin H_J^2(\mathcal P)$. We now consider the $J^2$-weighted norm. We have
\[
\|s^n\|_{H_{J^2}^2(\mathcal P)}^2
=
\frac16
\int_{\mathbb T^2}
|\lambda_1+\lambda_2|^{2n}
|\lambda_1-\lambda_2|^4
\,dm_2(\lambda_1,\lambda_2)
=
\frac{2}{3}\cdot 4^{n+1}
\int_{\mathbb T}
\cos^{2n}\frac{\theta}{2}
\sin^4\frac{\theta}{2}
\,dm_1(\theta).
\]
Another application of the beta integral gives
\[
\|s^n\|_{H_{J^2}^2(\mathcal P)}^2
=
\frac{2\cdot4^n\Gamma(n+1/2)}
{\sqrt{\pi}\Gamma(n+3)}.
\]
By a similar argument employed to obtain the lower bound of $\|q_n\|^2_{H_J^2(\mathcal P)}$ as above, one can find $C>0$ such that 
\begin{align*}
\|q_n\|_{H_{J^2}^2(\mathcal P)}^2
&=
\frac{2\Gamma(n+3/2)^2\Gamma(n+1/2)}
{8\Gamma(3/2)^2\sqrt{\pi}\,
\Gamma(n+1)^2\Gamma(n+3)} \leq C n^{-3/2}
\end{align*} 
for sufficiently large $n$. Consequently, 
$
\sum_{n=0}^{\infty}
\|q_n\|_{H_{J^2}^2(\mathcal P)}^2
<\infty
$
and so, $f\in H_{J^2}^2(\mathcal P)$. We have therefore obtained the strict inclusion
\[
H^2(\mathcal P) \subsetneq 
H_J^2(\mathcal P)
\subsetneq
H_{J^2}^2(\mathcal P).
\]
Next, we show that  $\|.\|_{H_J^2(\mathcal P)}$ and
$\|.\|_{H_{J^2}^2(\mathcal P)}$ are not equivalent norms on $H^2_J(\mathcal P)$. To see this, let $f_n(a,s,p)=s^n$ for $n \geq 1$. Then
\begin{align*}
    \frac{\|f_n\|^2_{H_{J^2}^2(\mathcal P)}}
{\|f_n\|^2_{H_J^2(\mathcal P)}}
=\displaystyle\left(\frac{2\cdot4^n\Gamma(n+1/2)}
{\sqrt{\pi}\Gamma(n+3)}\right)\displaystyle \left(\frac{4^n\Gamma(n+1/2)}
{\sqrt{\pi}\Gamma(n+2)}\right)^{-1}=\frac{2}{n+2} \to 0
\end{align*}
as $n \to \infty$. Next, we explore the relation of $H^2_{J^2}(\mathcal P)$ with the symmetrized bidisc. Let $g$ be holomorphic on $\mathbb G_2$ and define $\widetilde g(a,s,p)=g(s,p)$. Then $\widetilde g \in \text{Hol}(\mathcal P)$ and $
\widetilde g\circ\Psi_r=g\circ\pi_{\text{symm}, r}$ for $r\in (0, 1)$. Consequently,
\begin{align*}
\|\widetilde g\|_{H_{J^2}^2(\mathcal P)}^2
&=
\frac16 \sup_{0<r<1}r^4
\int_{\mathbb T^2}
|g\circ\pi_{\text{symm}, r}(\lambda_1, \lambda_2)|^2
|\lambda_1-\lambda_2|^4
\,dm_2.
\end{align*}
Thus the restriction of the Jacobian-squared weighted space to the
natural slice $\{0\}\times\mathbb G_2$ is not the usual Hardy space
$H^2(\mathbb G_2)$ defined as in \cite{MisraShyamRoyZhang2013} with the norm
\[
\|g\|_{H^2(\mathbb G_2)}^2
=
\frac{1}{2}
\sup_{0<r<1}r^2
\int_{\mathbb T^2}
\left|g\circ\pi_{\mathrm{symm},r}(\lambda_1, \lambda_2)\right|^2
\left|\lambda_1-\lambda_2\right|^2
\,dm_2.
\]
However, we get a higher-order Jacobian-weighted Hardy norm on
$\mathbb G_2$. This distinction is important. The weight $|J_\Psi|$ is precisely the one for which the natural slice recovers the standard Hardy space of the
symmetrized bidisc, whereas the weight $|J_\Psi|^2$ produces the
fourth-power Jacobian weight $|J_\pi|^4$ on that slice. Putting everything together, we have the following result.

\begin{theorem}\label{thm:comparison-weighted-Hardy}
Let $H^2(\mathcal P)$, $H_J^2(\mathcal P)$ and
$H_{J^2}^2(\mathcal P)$ denote the Hardy type Hilbert spaces associated with the weights $1$, $|J_\Psi|$ and $|J_\Psi|^2$, respectively.
Then the following hold.

\begin{enumerate}[leftmargin=*]
    \item There exist continuous inclusions
    $
    H^2(\mathcal P)
    \hookrightarrow
    H_J^2(\mathcal P)
    \hookrightarrow
    H_{J^2}^2(\mathcal P)$ and   
    \[
    \sqrt{3}\|f\|_{H_{J^2}^2(\mathcal P)} \leq 2\|f\|_{H^2_{J}(\mathcal P)} \leq 2 \sqrt{2} \|f\|_{H^2(\mathcal P)}
    \]
    for every $f \in H^2(\mathcal P)$.
    
    \item The inclusions are proper, that is, $
    H^2(\mathcal P)
    \subsetneq
    H_J^2(\mathcal P)
    \subsetneq
    H_{J^2}^2(\mathcal P)$. In particular, 
    \[
    \frac{1}{(2-s)} \in H_J^2(\mathcal P) \setminus H^2(\mathcal P) \quad \text{and} \quad \frac{1}{(2-s)^{3/2}} \in H_{J^2}^2(\mathcal P) \setminus H_{J}^2(\mathcal P).
    \]
   
    \item None of the pairs of norms corresponding to the three
    spaces is equivalent. Indeed,
    \[
    \frac{\|s^n\|_{H_J^2(\mathcal P)}}
    {\|s^n\|_{H^2(\mathcal P)}}
    =
    \frac{1}{\sqrt{n+1}} \quad \text{and} \quad 
    \frac{\|s^n\|_{H_{J^2}^2(\mathcal P)}}
    {\|s^n\|_{H_J^2(\mathcal P)}}
    =
    \sqrt{\frac{2}{n+2}}.
    \]

    \item The map $V:H^2(\mathbb G_2)\to H_J^2(\mathcal P)$ defined by
    $(Vg)(a,s,p)=g(s,p)$ is an isometry. 
\end{enumerate}
\end{theorem}

\noindent \textbf{Remarks.} The three Hilbert spaces $H^2(\mathcal P), H^2_J(\mathcal P)$ and $H^2_{J^2}(\mathcal P)$ considered here represent three different choices of boundary measure on $\mathcal P$. The first is the unweighted Hardy space  and the remaining two are Jacobian-weighted Hardy-type
spaces reflecting different aspects of the geometry of $\mathcal P$ and the parametrization $\Psi$. The two Jacobian constructions therefore have different geometric
features. The space $H_J^2(\mathcal P)$, obtained using the weight
$|J_\Psi|$, is distinguished by the fact that its natural
$\{a=0\}$ slice contains an isometric copy of $H^2(\mathbb G_2)$. On the
other hand, $H_{J^2}^2(\mathcal P)$ uses the square of the modulus of the
complex Jacobian and is consequently more directly analogous to the
usual Jacobian-squared measures arising in several complex variables. Among the three spaces considered here, we refer to $H^2(\mathcal P)$ as the Hardy space of the pentablock. This terminology is motivated by the fact that $H^2(\mathcal P)$ is defined directly from the boundary measure $\mu=\Psi_*m_3$ on the distinguished boundary $b\mathcal P$ without introducing an additional weight depending on the parametrization. This is particularly natural in the present
setting, since the boundary parametrization $\Psi$ is not a proper
finite-to-one map. Indeed, although it is generically two-to-one, its
fibres become positive-dimensional on the discriminant set
$\lambda_1=\lambda_2$, where
\[
    \Psi(\xi,\lambda,\lambda)=(0,2\lambda,\lambda^2)
\]
is independent of $\xi$. Thus, unlike the proper finite-to-one
parametrizations occurring in the standard constructions for domains such as $\mathbb G_2$ \cite{MisraShyamRoyZhang2013} and $\mathbb E$ \cite{JainKumarMalPramanick2026}, there is no
change-of-variables principle that forces a Jacobian factor into the boundary measure.

\section{Operators associated with the pentablock via the map $\Psi$}\label{sec_06}

\noindent We studied in Section \ref{sec_03} the map $\Psi: \overline{\mathbb D}^3 \to \overline{\mathcal P}$ at the scalar level, emphasizing its geometric and function-theoretic consequences for $\mathcal P$ and the Hardy space $H^2(\mathcal P)$. In this section, we study the operator-theoretic implications of $\Psi$. To begin with we recall from \cite{PalTomar2026} the notions of pentablock contractions and dilation.

\begin{definition}
    A commuting triple $(A, S, P)$ of operators acting on a Hilbert space $\mathcal H$ is said to be  a \textit{$\mathcal{P}$-contraction} if 
    \[
    \|q(A, S, P)\| \leq \|q\|_{\infty, \overline{\mathcal P}}
    \]
for every $q \in \mathbb{C}[a, s, p]$. We say that a $\mathcal P$-contraction $(A, S, P)$ acting on a Hilbert space $\mathcal H$ dilates to a $\mathcal P$-unitary if there exists a $\mathcal P$-unitary $(N_1, N_2, N_3)$ on a Hilbert space $\mathcal K \supseteq \mathcal H$ such that
  \[
  q(A, S, P)=P_{\mathcal H}q(N_1, N_2, N_3)|_{\mathcal H}
  \]
  for every $q \in \mathbb C[a, s, p]$. Here, $P_{\mathcal H}$ denotes the orthogonal projection of $\mathcal K$ onto $\mathcal H$.
\end{definition}

Various characterizations of $\mathcal P$-unitaries were presented in \cite{JindalKumar2024, PalTomar2026}. As mentioned in Section \ref{sec_02}, $b\mathcal P=\Psi(\mathbb T^3)$. Our next result is a generalization of this result at the level of operators, providing a new description of $\mathcal P$-unitaries. The proof follows an argument analogous to one of the characterizations of unitaries associated with the symmetrized bidisc given in \cite{AglerYoung1999}.

\begin{theorem}\label{thm_P_unitary}
A commuting triple $(N_1, N_2, N_3)$ of operators acting on a Hilbert
space $\mathcal H$ is a $\mathcal P$-unitary if and only if there exist commuting unitaries $U_1,U_2,U_3$ on $\mathcal H$
    such that $\Psi(U_1, U_2, U_3)=(N_1, N_2, N_3)$, i.e., 
    \[
    N_1=\frac{U_1(U_2-U_3)}{2},\qquad
    N_2=U_2+U_3,\qquad
    N_3=U_2U_3.
    \]
\end{theorem}

\begin{proof} Suppose $\underline{N}=(N_1,N_2,N_3)$ is a $\mathcal P$-unitary. By the joint spectral theorem, there is a unique spectral measure $E$ on $\sigma_T(\underline N)$ such that
\[
N_j=\int_{\sigma_T(\underline N)} z_j\,dE(z),\qquad j=1,2,3.
\]
The parametrization map $\Psi:\mathbb T^3 \to b\mathcal P$ is a continuous surjective map. Since $\mathbb T^3$ and $b\mathcal P$ are compact spaces,
$\Psi$ admits a Borel measurable right inverse $\rho:b\mathcal P\to \mathbb T^3$ satisfying $\Psi\circ\rho=\operatorname{id}_{b\mathcal P}$. Write $\rho=(\rho_1,\rho_2,\rho_3)$ and $U_j=\rho_j(N_1, N_2, N_3)$ for $1 \leq j \leq 3$, where $U_j$ is defined via the Borel functional calculus. Since $|\rho_j|=1$ on $\sigma_T(\underline N)$, each $U_j$ is a unitary. Moreover, $U_1, U_2, U_3$ commute as they belong to the abelian von Neumann algebra generated by $N_1, N_2, N_3$. By Borel functional calculus,
\[
(N_1, N_2, N_3)
=
\left(
\left(\frac{\rho_1(\rho_2-\rho_3)}{2}\right)(\underline N), (\rho_2+\rho_3)(\underline N), (\rho_2\rho_3)(\underline N)
\right)=\Psi(U_1, U_2, U_3).
\]
Conversely, let $U_1,U_2,U_3$ be commuting unitaries and $(N_1, N_2, N_3)=\Psi(U_1, U_2, U_3)$. Then $N_1, N_2, N_3$ are commuting normal operators. By spectral mapping theorem,
$
\sigma_T(N_1,N_2,N_3)
=
\Psi\bigl(\sigma_T(U_1,U_2,U_3)\bigr)
\subseteq
\Psi(\mathbb T^3)
=
b\mathcal P,
$
which completes the proof. 
\end{proof}

The following result is a simple application of Theorem \ref{thm_P_unitary}.

\begin{corollary}
    Let $(V_1, V_2, V_3)$ be a commuting triple of isometries acting on a Hilbert space $\mathcal H$. Then the triple $\Psi(V_1, V_2, V_3)$ is a $\mathcal P$-isometry.
\end{corollary}

\begin{proof}
It is a well-known fact that every commuting $n$-tuple of isometries admits an extension to a commuting $n$-tuple of unitaries. Let $(U_1, U_2, U_3)$ be a commuting tuple of unitaries on a Hilbert space $\mathcal K \supseteq \mathcal H$ such that $\mathcal H$ is invariant under $U_1, U_2, U_3$ and $(V_1, V_2, V_3)=(U_1|_{\mathcal H}, U_2|_{\mathcal H}, U_3|_{\mathcal H})$. By Theorem \ref{thm_P_unitary}, $\Psi(U_1, U_2, U_3)$ is a $\mathcal P$-unitary. Evidently, $\Psi(V_1, V_2, V_3)=\Psi(U_1, U_2, U_3)|_{\mathcal H}$. The proof is complete.
\end{proof}

The direct construction of $\mathcal P$-contractions is nontrivial. The following result gives a canonical way to present such examples from the map $\Psi$. Also, it presents a class of $\mathcal P$-contractions that admits dilation to $\mathcal P$-unitaries.

\begin{proposition}
    Let $(T_1, T_2, T_3)$ be a commuting triple of contractions on a Hilbert space $\mathcal H$ such that $\overline{\mathbb D}^3$ is a spectral set for $(T_1, T_2, T_3)$. Then $\Psi(T_1, T_2, T_3)$ is a $\mathcal P$-contraction. In particular, if $(T_1, T_2, T_3)$ admits a dilation to a commuting triple of unitaries, then $\Psi(T_1, T_2, T_3)$ admits a dilation to a $\mathcal P$-unitary.
\end{proposition}

\begin{proof}
 Let $q \in \mathbb C[a, s, p]$. Then $q\circ \Psi$ is a holomorphic function on $\overline{\mathbb D}^3$. Since $(T_1, T_2, T_3)$ has $\overline{\mathbb D}^3$ as a spectral set and $\Psi(\overline{\mathbb D}^3) \subseteq \overline{\mathcal P}$, it follows that
    \begin{align*}
        \|q(\Psi(T_1, T_2, T_3))\|=\|(q\circ \Psi)(T_1, T_2, T_3)\| \leq \|q\circ \Psi\|_{\infty, \overline{\mathbb D}^3} \leq \|q\|_{\infty, \overline{\mathcal P}}.
    \end{align*}
Consequently, $\Psi(T_1, T_2, T_3)$ is a $\mathcal P$-contraction. Suppose $(T_1, T_2, T_3)$ admits a dilation to a commuting tuple of unitaries $(U_1, U_2, U_3)$ on a Hilbert space $\mathcal K \supseteq \mathcal H$. Let $q \in \mathbb C[a, s, p]$ and let $g=q \circ \Psi$. Then $g \in \text{Hol}(\overline{\mathbb D}^3)$ and 
\[
q(\Psi(T_1, T_2, T_3))=g(T_1, T_2, T_3)=P_{\mathcal H}g(U_1, U_2, U_3)|_{\mathcal H}=P_{\mathcal H}q\left(\Psi(U_1, U_2, U_3)\right)
|_{\mathcal H}.
\]
By Theorem \ref{thm_P_unitary}, $\Psi(U_1, U_2, U_3)$ is a $\mathcal P$-unitary, which completes the proof.
\end{proof}

We have proved in Proposition \ref{prop:Psi-interior} that $\Psi(\overline{\mathbb D}^3)$ is a strict subset of $\overline{\mathcal P}$. Consequently, the converse to the above proposition is not true, that is, not every $\mathcal P$-contraction is of the form $\Psi(T_1, T_2, T_3)$, where $(T_1, T_2, T_3)$ is a commuting triple of contractions with $\overline{\mathbb D}^3$ as its spectral set.

\section*{Appendix}
\noindent Recall from Section \ref{sec_02} that the radial Hardy space $H^2_{\mathrm{rad}}(\mathcal{P})$ is given by
\[
H^2_{\mathrm{rad}}(\mathcal{P})
=
\left\{
f\in\operatorname{Hol}(\mathcal{P}):
\|f\|_{H^2_{\mathrm{rad}}(\mathcal{P})}=
\sup_{0<r<1}
\left(\int_{b\mathcal{P}}
|f\circ\delta_r|^2\,d\mu\right)^{1\slash 2}
<\infty
\right\},
\]
where $\delta_r(a, s, p)=(ra, rs, r^2p)$ for $0<r<1$. We have proved in Section \ref{sec_02} that the space $H^2(\mathcal{P})$ is isometrically isomorphic to $H^2_{\mathrm{rad}}(\mathcal P)$. Consequently, the polynomial algebra is dense in $H^2_{\mathrm{rad}}(\mathcal P)$. Also, Corollary \ref{lem_rad} states that $\|f-f\circ \delta_r\|_{H^2(\mathcal P)} \to 0$ as $r \uparrow 1$ for every $f \in H^2(\mathcal P)$.  We present alternative proofs to these facts using only the radial Hardy space norm.

\begin{lemma}[Radial approximation]\label{lem_radII}
Let $f \in H_{\mathrm{rad}}^2(\mathcal{P})$ and define $f_r=f\circ \delta_r$ for $r \in (0, 1)$. Then $\lim_{r\to 1}\|f_r-f\|_{H_{\mathrm{rad}}^2(\mathcal{P})} = 0$.
\end{lemma}

\begin{proof}
Consider the parametrization map $\Psi$ and the map $\Psi_r$ given by
\[
\Psi(\xi, \lambda_1, \lambda_2)=\left(
\frac{\xi(\lambda_1-\lambda_2)}{2},
\lambda_1+\lambda_2,
\lambda_1 \lambda_2
\right) \quad \text{and} \quad \Psi_r(\xi, \lambda_1, \lambda_2)=\Psi(\xi, r\lambda_1, r\lambda_2)
\]
for $0<r<1$. Define $F: \overline{\mathbb D} \times \mathbb D^2 \to \mathcal P$ as $F=f \circ \Psi$. Since $\Psi(\overline{\mathbb D} \times \mathbb{D}^2)\subseteq\mathcal{P}$, the map $F$ is
holomorphic on $\mathbb{D}^3$ and continuous on $\overline{\mathbb D} \times \mathbb D^2$. For $0<r<1$, define
$F_r(\xi,\lambda_1, \lambda_2)=F(\xi,r \lambda_1,r \lambda_2)$. Then
\[
\begin{aligned}
F_r(\xi,\lambda_1,\lambda_2)
=
f\left(
\frac{r\xi(\lambda_1-\lambda_2)}{2},
r(\lambda_1+\lambda_2),
r^2\lambda_1\lambda_2
\right)
=f\circ\Psi_r(\xi,\lambda_1,\lambda_2).
\end{aligned}
\]
Following the proof for Proposition \ref{prop:isometric-embedding} for $\|.\|_{H^2_{\mathrm{rad}}(\mathcal P)}$, one can show that $
\|F\|_{H^2(\mathbb D^3)}=\|f\|_{H^2_{\mathrm{rad}}(\mathcal{P})}$ and so, $F\in H^2(\mathbb{D}^3)$. Let us write the
Taylor expansion of $F$ on $\mathbb D^3$ as
\[
F(\xi, \lambda_1, \lambda_2)
=
\sum_{m,n,k\geq0}
c_{m,n,k}\xi^m \lambda_1^n \lambda_2^k.
\]
Then 
$
\|f_r-f\|_{H^2_{\mathrm{rad}}(\mathcal{P})}^2=\|F_r-F\|_{H^2(\mathbb{D}^3)}^2
=
\sum_{m,n,k\geq0}
|c_{m,n,k}|^2|r^{n+k}-1|^2 \to 0
$
as $r \to 1$. Here, the convergence to $0$ follows from the dominated convergence theorem since $0<r<1$ and $\|F\|^2_{H^2(\mathbb D^3)}=$$\sum_{m,n,k\geq0}|c_{m,n,k}|^2<\infty$.
\end{proof}

\begin{theorem}
The polynomial algebra is dense in $H_{\mathrm{rad}}^2(\mathcal P)$.
\end{theorem}

\begin{proof}
Evidently, $\overline{\mathbb{C}[a,s,p]}^{\,\|\cdot\|_{H_{\mathrm{rad}}^2(\mathcal{P})}}
\subseteq H_{\mathrm{rad}}^2(\mathcal{P})$. Conversely, let $f\in H_{\mathrm{rad}}^2(\mathcal{P})$ and define $f_n=f\circ \delta_{\alpha_n}$, where $\alpha_n=1-1/n$ for $n \geq 1$. Since $\mathcal{P}$ is $(1,1,2)$-quasi-balanced, each $f_n$ is holomorphic on a neighbourhood of $\overline{\mathcal{P}}$. As $\overline{\mathcal{P}}$ is polynomially convex, the
Oka-Weil gurantees the existence of a sequence of polynomials
$\{p_{n,k}\}_{k=1}^{\infty}$ such that $\|p_{n,k}- f_n\|_{\infty, \overline{\mathcal P}} \to 0$ as $k \to 0$. Then for every $n, k \geq 1$, 
\[
\|p_{n,k}-f_n\|_{H_{\mathrm{rad}}^2(\mathcal{P})}^2
=
\sup_{0<r<1}
\int_{\mathbb{T}^3}
|(p_{n,k}-f_n)\circ\Psi_r|^2\,dm
\leq
\|p_{n,k}-f_n\|_{\infty, \overline{\mathcal{P}}}^2
\]
and so, $p_{n, k} \to f_n$ in $H^2_{\mathrm{rad}}(\mathcal P)$
Thus, $f_n\in
\overline{\mathbb{C}[a,s,p]}^{\,\|\cdot\|_{H_{\mathrm{rad}}
^2(\mathcal{P})}}$ for every $n \geq 1$. By Lemma \ref{lem_radII}, $\|f_n-f\|_{H^2_{\mathrm{rad}}(\mathcal P)} \to 0$ as $n \to \infty$. Since each $f_n \in \overline{\mathbb{C}[a,s,p]}^{\,\|\cdot\|_{H^2_{\mathrm{rad}}(\mathcal{P})}}$, it follows that
$f$ belongs to the same space. Hence $H^2_{\mathrm{rad}}(\mathcal{P})
=
\overline{\mathbb{C}[a,s,p]}^{\,\|\cdot\|_{H^2_{\mathrm{rad}}(\mathcal{P})}}$.
\end{proof}

\noindent \textbf{Acknowledgments.} 
Guicong Su was partially supported by the National Natural Science Foundation of China (Nos.\ 1231101308, \ 12671102) and the Natural Science Foundation of Shandong Province (No.\ ZR2025MS77).
Nitin Tomar was  supported via the Postdoctoral Fellowship (Ref. No. 0204/9(10)/2026-R\&D-II/4805) awarded by the National Board for Higher Mathematics (NBHM), Government of India.

\bibliographystyle{plain}   
\bibliography{references} 
\end{document}